\def\AUTHORSIN{1}

\documentclass[reqno,a4paper,12pt]{amsart} 

\usepackage{amsmath,amscd,amsfonts,amssymb}
\usepackage{mathrsfs,dsfont}

\usepackage{xcolor,graphicx}

\usepackage[normalem]{ulem}

\ifdefined\MIHALIS  %%%% The next few lines change the appearance of the paper
\usepackage{txfonts,pxfonts,tikz} %also txfonts 
\usepackage{tgschola}
\usepackage[T1]{fontenc}
\usepackage[allcolors=blue,allbordercolors=blue,pdfborderstyle={/S/U/W 1}]{hyperref}
\usepackage[ocgcolorlinks]{ocgx2}
\usepackage{bbm}

\ifdefined\SMART %%%% The next few lines change the appearance of the paper
\usepackage[paperwidth=10cm,paperheight=20cm,left=5mm,right=5mm,top=10mm,bottom=5mm]{geometry}
\else
\calclayout
\fi

\fi % MIHALIS
\newcommand{\Set}[1]{{\left\{{#1}\right\}}}

\newcommand{\RR}{{\mathbb R}}

\renewcommand{\dim}{{\rm dim\,}}

\numberwithin{equation}{section}
\numberwithin{figure}{section}

\newcommand\R{\mathbb{R}}

\newcommand\Z{\mathbb{Z}}

\newcommand\al{\alpha}
\newcommand\be{\beta}
\newcommand\gam{\gamma}

\newcommand\del{\delta}
\newcommand\Del{\Delta}

\newcommand\lam{\lambda}
\newcommand\Lam{\Lambda}

\newcommand\1{\mathds{1}}
\newcommand\eps{\varepsilon}

\newcommand\pphi{\varphi}

\renewcommand\le{\leqslant}
\renewcommand\ge{\geqslant}
\renewcommand\leq{\leqslant}
\renewcommand\geq{\geqslant}
\newcommand\sbt{\subset}
\newcommand\tint{{\textstyle \int}}

\renewcommand\hat{\widehat}

\renewcommand\S{S}

\newcommand{\ft}[1]{\widehat{#1}}
\newcommand{\dotprod}[2]{\langle #1 , #2 \rangle}

\newcommand{\supp}{\operatorname{supp}}

\newcommand{\cm}{\complement}

\theoremstyle{plain}
\newtheorem{thm}{Theorem}[section]
\newtheorem{lem}[thm]{Lemma}
\newtheorem{lemma}[thm]{Lemma}
\newtheorem{cor}[thm]{Corollary}

\newtheorem{prop}[thm]{Proposition}

\newtheorem*{claim*}{Claim}

\newcommand{\thmref}[1]{Theorem~\ref{#1}}
\newcommand{\secref}[1]{Section~\ref{#1}}

\newcommand{\lemref}[1]{Lemma~\ref{#1}}

\newcommand{\propref}[1]{Proposition~\ref{#1}}

\theoremstyle{definition}
\newtheorem{definition}[thm]{Definition}
\newtheorem*{definition*}{Definition}
\newtheorem*{remarks*}{Remarks}
\newtheorem*{remark*}{Remark}

\newenvironment{enumerate-roman}
{\begin{enumerate}
\addtolength{\itemsep}{5pt}
}
{\end{enumerate}}

\newenvironment{enumerate-alph}
{\begin{enumerate}
\addtolength{\itemsep}{5pt}
}
{\end{enumerate}}

\newenvironment{enumerate-num}
{\begin{enumerate}
\addtolength{\itemsep}{5pt}
}
{\end{enumerate}}

\newenvironment{enumerate-text}
{\begin{enumerate}
\addtolength{\itemsep}{5pt}
}
{\end{enumerate}}

\begin{document}

\title{The Tur\'{a}n and Delsarte problems and their duals}

\ifdefined\AUTHORSIN

\author[M. Kolountzakis]{Mihail N. Kolountzakis}
\address{Department of Mathematics and Applied Mathematics, University of Crete, Voutes Campus, 70013 Heraklion, Greece and Institute of Computer Science, Foundation of Research and Technology Hellas, N. Plastira 100, Vassilika Vouton, 700 13, Heraklion, Greece}
\email{kolount@uoc.gr}

\author[N. Lev]{Nir Lev}
\address{Department of Mathematics, Bar-Ilan University, Ramat-Gan 5290002, Israel}
\email{levnir@math.biu.ac.il}

\author[M. Matolcsi]{M\'at\'e Matolcsi}
\address{HUN-REN Alfr\'ed R\'enyi Institute of Mathematics, Re\'altanoda utca 13-15, H-1053, Budapest, Hungary and Department of Analysis and Operations Research, Institute of Mathematics, Budapest University of Technology and Economics, M\H uegyetem rkp. 3., H-1111 Budapest, Hungary}
\email{matomate@renyi.hu}

\fi

\date{October 11, 2025}
\subjclass{42B10, 43A35, 52A41, 52C22}
\keywords{Tur\'{a}n problem, Delsarte problem, linear duality, packing, tiling}

\ifdefined\AUTHORSIN
\thanks{N.L.\ is supported by ISF Grant No.\ 1044/21.
M.M.\ is supported by the Hungarian National Foundation for Scientific Research, Grants No. K146387 and KKP 133819.}
\fi
\dedicatory{Dedicated to the memory of Jean-Pierre Gabardo}

\begin{abstract}
We study two optimization problems for positive definite functions on Euclidean space with restrictions on their support and sign: the Tur\'an problem and the Delsarte problem. These problems have been studied also for their connections to geometric problems of tiling and packing. In the finite group setting the weak and strong linear duality for these problems are automatic. We prove these properties in the continuous setting. We also show the existence of extremizers for these problems and their duals, and establish tiling-type relations between the extremal functions for each problem and the extremal measures or distributions for the dual problem. We then apply the results to convex bodies, and prove that the Delsarte packing bound is strictly better than the trivial volume packing bound for \emph{every} convex body that does not tile the space.
\end{abstract}

\maketitle

\ifdefined\MIHALIS
\tableofcontents
\fi

% ================================================

\section{Introduction}

The Tur\'an extremal problem on positive definite functions with restricted support (see \cite{Ste72} as well as the history and references in \cite{Rev11}) is a very natural question asking how large the integral of a positive definite function $f$ can be if its support is restricted to a domain $U$ and its value at the origin is normalized to be $f(0)=1$. The problem makes sense in every locally compact abelian (LCA) group and it is interesting in many such groups, including Euclidean spaces (our main focus here) and finite groups. It is easy to see that it suffices to consider continuous positive definite functions, or equivalently,  functions $f$ whose Fourier transform $\ft f$  is everywhere nonnegative on the dual group (see \cite[Section 1.4]{Rud62}).

If $U$ is a bounded origin-symmetric open set,
and  $A$ is a measurable set  of positive 
measure  such that $A-A \subset U$,
then the function $f = m(A)^{-1} \1_A \ast \1_{-A}$
is supported in $U$, is positive definite and
$f(0)=1$, hence the supremum in question is at least
as large as $\int f = m(A)$. 
A natural question is therefore to decide if the
\emph{Tur\'an constant} of $U$, defined as
\begin{equation}
\label{turan-constant}
T(U) = \sup \Big\{\int  f:
\text{$f(0)=1$, \, $f = 0$ on $U^\cm$, \, $\ft{f}$ is nonnegative} \Big\},
\end{equation}
is equal to
\begin{equation}\label{supma}
\sup\Set{m(A): A-A \subset U}.
\end{equation}
It is not hard to see that the answer is negative in this generality, 
i.e.\ there exist sets $U$ such that \eqref{turan-constant} is strictly greater than \eqref{supma} (examples in finite groups are easy to construct, see e.g.\ \cite{MR14}, and these examples can be carried over to the Euclidean setting in an obvious manner by placing small cubes around points).  

However, interestingly, the question remains open if we restrict ourselves to convex domains. If $U \sbt \R^d$ is an origin-symmetric convex body, then it is a consequence of the Brunn-Minkowski inequality that the quantity \eqref{supma} is the one that corresponds to the set $A= \frac12 U$ and hence is equal to $2^{-d}m(U)$. A major open problem is therefore to decide whether $T(U) = 2^{-d} m(U)$ for every  origin-symmetric convex body $U$. We say that $U$ is a \emph{Tur\'an domain} if it satisfies this equality. It is known that any convex body $U$ that tiles the space by translations, as well as the Euclidean ball in $\RR^d$, is a Tur\'an domain (and of course their linear images are too) \cite{Gor01, AB02, KR03}. It is also not hard to see that cartesian products of Tur\'an domains are again Tur\'an domains, but we know the quantity $T(U)$ for no other origin-symmetric convex body. Even the simplest cases, such as a regular octagon in the plane, are open. 

A related quantity to the Tur\'an constant is the {\em Delsarte constant} of $U$, defined as
\begin{equation}\label{delsarte-constant}
D(U) = \sup \Big\{\int f : \text{$f(0)=1$, \, $ f\le 0$ on $U^\cm$, \,
$\ft{f}$ is nonnegative} \Big\},
\end{equation}
so that the condition $f=0$ on $U^\cm$ in the Tur\' an problem is now replaced with $f\le 0$ on $U^\cm$. The Delsarte problem, introduced in \cite{Del72,DGS77}, has found many applications in estimating geometric quantities such as sphere packing densities, kissing numbers and more (see \cite{CE03,Via17,CKMRV17,CLS22} and the references in \cite{BR23}).

The Tur\'an and Delsarte extremal problems are optimization problems where we seek the optimum of a linear functional over a function space defined by a set of linear inequalities. In other words, they are {\em linear programs}, albeit infinite dimensional ones. Their \emph{dual} problems, called the \emph{dual Tur\'an} and \emph{dual Delsarte} problems, are therefore intimately connected to the \emph{primal} Tur\'an and Delsarte problems that we have defined above. These optimization problems are much better understood in the finite group setting, where, importantly, weak and strong duality are always valid. In \secref{secfin} we review the Tur\'an and Delsarte problems along with their duals in  finite groups, and establish tiling-type relations among the extremizers of these problems. We also connect the problems to the notions of tiling and spectrality in finite groups.

The extension of many of these results to the continuous setting (of domains in $\RR^d$) is not obvious, as in the setting of infinite dimensional linear programs duality may not hold \cite{AN87}. In the literature, results about weak and strong duality and the existence of extremizers have appeared under various conditions \cite{CLS22, BRR24, Gab24}. In this paper we present a unified treatment of the Tur\'an and the Delsarte problems and their duals in the Euclidean setting, and apply our results to obtain interesting consequences concerning the Delsarte packing bound, and properties of Tur\'an domains.

After some preliminaries in \secref{secP1}, we define the primal and dual Tur\'an and Delsarte problems in Sections \ref{s:turan} and \ref{s:delsarte} respectively. We first establish the corresponding weak linear duality inequalities, and then  proceed to prove that in fact strong linear duality  holds for both problems. We show that extremizers for the Tur\'an and Delsarte problems and their duals exist and satisfy tiling-type convolution equalities. 

In \secref{s:packing-tiling} we connect the Delsarte constant of the difference set $A-A$ to the density of packing by translated copies of a set $A \sbt \R^d$, and to tiling and spectrality properties of $A$. Finally, we apply our results to convex bodies. We prove that the Delsarte packing bound is strictly better than the trivial volume packing bound for every convex body $A$ that does not tile the space. We also give a possible path of attack for proving the existence of a convex body which is  not  a Tur\'an domain.

% =================================================

\section{The Tur\'an and Delsarte problems in finite groups}
\label{secfin}

We start our discussion in the context of finite abelian groups, which
motivates the forthcoming results in the Euclidean setting in the later sections.
The Tur\'an and Delsarte problems in finite groups were discussed
in detail in \cite{MR14}. We recall some relevant results for convenience. 

If $G$ is a finite abelian group and $\hat{G}$ is the dual group, 
then we define the Fourier transform of a function $f$ on $G$ as
\begin{equation}
\label{ffin}
\hat{f}(\gamma) = |G|^{-1/2} \sum_{x\in G} f(x)\gamma(-x), \quad \gam \in \hat{G},
\end{equation}
and the convolution of two functions $f$ and $h$ on $G$ as
\begin{equation}
\label{fhconv}
( f \ast h)(x) =  |G|^{-1/2} \sum_{y\in G} f(y) h(x-y), \quad x \in G.
\end{equation}
The Fourier transforms of 
$f \ast h$ and $f \cdot h$
are given by $ \ft{f} \cdot \ft{h}$
and $ \ft{f} \ast \ft{h}$ respectively.

\subsection{The Tur\'an problem}
Let $U \sbt G$ be a set with $0 \in U = - U$, i.e.\ $U$ is origin-symmetric
and contains the origin. A function $f$ on $G$
 is called  Tur\'an admissible if $f$ is a real-valued function
 supported on $U$, 
 $f(0)=1$  and $\hat{f}$ is nonnegative. 
 The \emph{Tur\'an constant} $T(U)$ 
 is the supremum of $ \ft{f}(0)$
 over all the Tur\'an admissible functions $f$. 

 We say that a
  function $h$ on $G$ is admissible for the 
 dual Tur\'an problem if $h(0)=1$, $h$ vanishes on $U \setminus  \{0\}$, 
 and $\hat{h}$ is nonnegative. 
The \emph{dual Tur\'an constant} $T'(U)$ is the  
supremum of $ \ft{h}(0)$ 
over all the  dual Tur\'an  admissible functions $h$.

It is easy to check using Plancherel's theorem 
that the inequality
$T(U) T'(U) \le 1$ holds, which is
referred to as  weak linear duality.
It was proved in \cite[Theorem 4.2]{MR14}
that moreover,  there is a strong linear duality:

\begin{thm}
 $T(U)T'(U)=1$ holds for every   $U \sbt G$ with
 $0 \in U = - U$.
\end{thm}

Since the set of Tur\'an admissible functions is compact and the mapping
$f \mapsto \ft{f}(0)$ is continuous, there exists at least one
extremal function $f$ for  the Tur\'an problem.
Similarly, also  the dual  Tur\'an problem
admits an extremal function $h$.

\begin{prop}
If $f$ and $h$ are 
extremal functions for 
 the Tur\'an problem and its dual
 respectively, 
 then $f \ast h = |G|^{-1/2}$ and $\ft{f} \cdot \ft{h} = \del_0$.
\end{prop}

Here and below we use $\del_0$ to denote the
function which takes the value $1$ at the origin
and vanishes everywhere else.
To prove the proposition we observe that
\begin{equation}
\label{eq:F.2.1}
1 =  \sum_{x \in G} f(x) \overline{h(x)}
 =  \sum_{\gam \in \hat{G}} \ft{f}(\gam) \ft{h}(\gam)
 \ge \ft{f}(0) \ft{h}(0) = T(U) T'(U) = 1,
\end{equation}
and as a consequence, the 
inequality in \eqref{eq:F.2.1} is in fact an equality.
Hence $ \ft{f}(\gam) \ft{h}(\gam) = 0$ for all nonzero $\gam \in \hat{G}$,
which implies the claim.

\subsubsection*{Remark}
We obviously also have
$f \cdot h = \del_0$
and therefore $\ft{f} \ast \ft{h} = |G|^{-1/2}$,
but this holds for arbitrary admissible $f$ 
and $h$, not only for extremal functions.

\subsection{The Delsarte problem}
We again assume that 
 $U \sbt G$ is a set with $0 \in U = - U$.
 A function $f$ on $G$
 is called Delsarte admissible if
 $f$ is a real-valued function
 such that $f(0)=1$, $f(x) \le 0$ for all $x \in G \setminus U$,
 and $\hat{f}$ is nonnegative. 
 The \emph{Delsarte constant} $D(U)$ 
 is the supremum of $ \ft{f}(0)$
 over all the Delsarte admissible functions $f$. 

 We say that a
  function $h$ on $G$ is admissible for the 
 dual Delsarte problem if $h(0)=1$, 
 $h$ vanishes on $U \setminus  \{0\}$, 
 $h$ is nonnegative on $G \setminus U$,  
 and $\hat{h}$ is  nonnegative everywhere. 
The \emph{dual Delsarte constant} $D'(U)$ is  the  
supremum of $ \ft{h}(0)$ 
over all the dual Delsarte admissible 
functions $h$.

It is again easy to check that the
weak linear duality inequality
$D(U) D'(U) \le 1$ holds, while 
\cite[Theorem 4.2]{MR14})
establishes that
there is a strong linear duality:

\begin{thm}
 $D(U)D'(U)=1$ holds for every   $U \sbt G$ with
 $0 \in U = - U$.
\end{thm}

The existence of 
extremal functions for the Delsarte
problem and its dual
 is again obvious.

\begin{prop}
If $f$ and $h$ are 
extremal functions for 
 the Delsarte problem and its dual
 respectively, 
 then
 \begin{enumerate-num}
 \item 
 \label{it:fde.1}
$f \cdot h = \del_0$,
and as a consequence, $\ft{f} \ast \ft{h} = |G|^{-1/2}$;
 \item
 \label{it:fde.2}
 $f \ast h = |G|^{-1/2}$,
 and 
$\ft{f} \cdot \ft{h} = \del_0$.
 \end{enumerate-num}
\end{prop}

To prove this we note that
\begin{align}
1 & = f(0)h(0) \ge  \sum_{x \in G} f(x) h(x)
 =  \sum_{\gam \in \hat{G}} \ft{f}(\gam) \ft{h}(\gam) \label{eq:F.3.1} \\
 & \ge  \ft{f}(0) \ft{h}(0) = D(U) D'(U) = 1, \label{eq:F.3.2}
\end{align}
hence both 
inequalities  in \eqref{eq:F.3.1} and
\eqref{eq:F.3.2} are equalities.
It follows that $f(x) h(x) =0$ for all nonzero $x \in G$,
while 
 $ \ft{f}(\gam) \ft{h}(\gam) = 0$ for all nonzero $\gam \in \hat{G}$,
 which implies both conclusions
 \ref{it:fde.1} and \ref{it:fde.2} above.

\subsection{Difference sets}

Let $A \sbt G$ be an arbitrary nonempty set,
and denote
\begin{equation}
m(A) = |G|^{-1/2} |A|.
\end{equation}
We now connect
the Tur\'an and Delsarte constants of the difference
set $U = A-A$ to 
 tiling and spectrality properties of the set $A$.
We first observe that the function
 $f = m(A)^{-1} \1_A \ast \1_{-A}$
is Tur\'an  admissible and
$ \ft{f}(0)= m(A)$, hence
\begin{equation}
\label{eq:f.d.4}
D(A-A) \ge T(A-A) \ge  m(A).
\end{equation}

We say that  $A$ 
\emph{tiles by translations} if there is a set
$\Lam \sbt G$ such that the 
translated copies $A + \lam$, $\lam \in \Lam$,
form a partition of $G$. 
We say that  $A$ is a \emph{spectral set} if
it admits 
a system of characters
$\Lam \sbt \hat{G}$
which forms an orthogonal basis for
the space $L^2(A)$.

\begin{prop}
\label{prop:f.d.1}
If $A$ either tiles or is spectral, then 
\begin{equation}
\label{eq:f.d.1}
D(A-A) = T(A-A) = m(A).
\end{equation}
\end{prop}

\begin{proof}
Assume first that $A$ tiles with a translation set
$\Lam \sbt G$. Then
\begin{equation}
\label{eq:f.d.2}
(A-A) \cap (\Lam - \Lam) = \{0\}, \quad 
|A| \cdot |\Lam| = |G|.
\end{equation}
Hence the function
 $h := m(A) \cdot \1_\Lam \ast \1_{ - \Lam }$
is dual Delsarte admissible with respect to the set $U = A-A$,
and satisfies
$ \ft{h}(0)  = m(A)^{-1}$.
Due to weak linear duality, this implies the inequality
$D(A-A) \le m(A)$.
Together with \eqref{eq:f.d.4} this implies
\eqref{eq:f.d.1}.

Next, suppose that $A$ is spectral,
i.e.\ there is  a  system of characters $\Lam
\sbt \hat{G}$ which
forms an orthogonal basis
for the space $L^2(A)$. Then
\begin{equation}
\label{eq:f.d.5}
 \Lam - \Lam\sbt \{ 
 \ft{\1}_A  = 0 \} \cup \{0 \}, \quad 
|\Lam| = \dim L^2(A) = |A|.
\end{equation}
Let $h$ be the function on $G$ defined by
\begin{equation}
\label{eq:f.d.6}
h(x) = |A|^{-2} \cdot \Big| \sum_{\gam \in \Lam} \gam(x) \Big|^2,
\quad x \in G,
\end{equation}
then both $h$ and $\ft{h}$ are nonnegative functions,
\begin{equation}
\label{eq:f.d.7}
 h(0)=1, \quad
\ft{h}(0) = m(A)^{-1},
\quad \supp \ft{h} \sbt \Lam-\Lam.
\end{equation}
This implies that
$\ft{\1}_A \cdot \ft{h} = \del_0$, and as a consequence,
 $\1_A \ast h = |G|^{-1/2}$.
Since $h(0)=1$ and $h$ is nonnegative,
this is possible only if
$h$ vanishes on  the set $(A-A) \setminus \{0\}$.
Hence 
 $h$ is dual Delsarte admissible with respect to $U=A-A$,
 so again by weak linear duality we obtain that
$D(A-A) \le m(A)$.  As before, 
together with \eqref{eq:f.d.4} this yields
\eqref{eq:f.d.1}.
\end{proof}

\subsubsection*{Remark}
The converse to \propref{prop:f.d.1} does not hold.
An example constructed in \cite{KLMS24} shows
that there is a finite abelian group $G$
and a set $A \sbt G$, such that
\eqref{eq:f.d.1} holds but 
$A$ neither tiles nor is spectral.

In the next sections we will see that adapting these results to the Euclidean setting is far from obvious.

% =========================================================

\section{Euclidean setting preliminaries}
\label{secP1}

In this section we recall some necessary background
in the Euclidean setting
and fix notation that will be used throughout the paper
(see also \cite{Rud91} for more details).

\subsubsection*{Notation}

If $A \sbt \R^d$ then $A^\cm = \R^d \setminus A$ is the complement of $A$,
$\partial A$ is the boundary of $A$, and $m(A)$ is the Lebesgue measure of $A$.

\subsection{}
The \emph{Schwartz space} $\S(\R^d)$ consists of all infinitely smooth
functions $\varphi$ on $\R^d$ such that for every $n$ and 
every multi-index $k = (k_1,\dots,k_d)$, the seminorm
\[
\|\varphi\|_{n,k} := \sup_{x\in \R^d} |x|^n |\partial^k \varphi(x)|
\]
is finite.  A \emph{tempered distribution} is a 
linear functional  on the Schwartz space  
which is continuous with respect to the 
topology generated by this family of seminorms.
We use  $\alpha(\varphi)$ to denote  the action of
a  tempered distribution $\alpha$ on a
Schwartz function $\varphi$.

We use $\S'(\R^d)$ to denote the 
space of tempered distributions on $\R^d$.
A sequence of tempered distributions $\al_j$  is said to
converge in the space $\S'(\R^d)$ if there exists 
a  tempered distribution $\al$
such that $\al_j(\pphi) \to \al(\pphi)$
for  every 
Schwartz function $\varphi$.

If $\varphi$ is a Schwartz function on $\R^d$ then its Fourier transform
is defined by
\[
\ft \varphi (t)=\int_{\R^d} \varphi (x)  e^{-2\pi i\langle t,x\rangle} dx, \quad t \in \R^d.
\]
The Fourier transform of a  tempered distribution 
$\alpha$ is defined by 
$\ft{\alpha}(\varphi) = \alpha(\ft{\varphi})$.

If $\alpha$ is a tempered distribution  and if
$\varphi$  is a 
Schwartz function, then 
the product $\alpha \cdot \varphi$ is 
a tempered distribution   defined
by $(\alpha \cdot \varphi)(\psi) =
\alpha(\varphi \cdot \psi)$,
$\psi \in \S(\R^d)$.
The convolution
$\alpha \ast \varphi$ 
of a tempered distribution  
 $\alpha$ and 
a Schwartz function $\varphi$ is 
 an infinitely smooth function
which is also
a tempered distribution, and
whose Fourier transform is $\ft{\al} \cdot \ft{\varphi}$.

A tempered distribution  $\al$  is called
\emph{real} if $\al(\pphi)$ is a real scalar
for every real-valued $\pphi \in \S(\R^d)$.
We say that  $\al$ is 
\emph{even} if   $\al(\pphi) = 0$ for every 
odd  $\pphi \in \S(\R^d)$.

\subsection{}
If $\mu$ is a locally finite (complex) measure on $\R^d$,
then we say that $\mu$
is a tempered distribution if there exists a 
 tempered distribution $\al_\mu$ satisfying
 $\al_\mu(\pphi) = \int \varphi d \mu$
 for every smooth function $\pphi$ with compact support.
 If such $\al_\mu$ exists, then it is unique.

A measure $\mu$ on $\R^d$ is  called \emph{translation-bounded}
if there exists a constant $C$ such that
$ |\mu|(B + t) \le C$ for every $t \in \R^d$,  where 
$B$ is the open unit ball in $\R^d$.
If a measure $\mu$ is  translation-bounded,
then it is a tempered distribution.

If $\mu$  is a translation-bounded measure  on $\R^d$, and if $\nu$
is a finite measure  on $\R^d$, then the convolution $\mu \ast \nu$
is a translation-bounded measure.

\begin{lem}[{see \cite[Section 2.5]{KL21}}]
\label{lemB1}
Let $\nu$ be a finite measure on $\R^d$, and let
$\mu$ be a translation-bounded measure on $\R^d$
whose Fourier transform $\ft{\mu}$ is a locally finite measure.
Then the Fourier transform of the convolution $\mu \ast \nu$
 is the measure $\ft{\mu} \cdot \ft{\nu}$.
\end{lem}

A sequence of measures $\{\mu_j\}$ is said to be
\emph{uniformly translation-bounded} if 
there exists a constant $C$ 
such that $\sup_t |\mu_j|(B+t) \leq C$ for all $j$,
where $B$ is again the open unit ball in $\R^d$.
If $\{\mu_j\}$ is a uniformly
  translation-bounded sequence of measures,
then $\mu_j$ is said to \emph{converge vaguely}
 to a measure $\mu$ if for
every continuous, compactly supported function $\varphi$ we have
$\int \varphi  d\mu_j \to \int \varphi  d\mu$. In this case, the 
vague limit  $\mu$
must also be a translation-bounded measure. 
From any uniformly translation-bounded sequence of 
measures $\{\mu_j\}$  one can extract
a vaguely convergent subsequence.

Similarly, a sequence of finite measures 
$\{\mu_j\}$ on $\R^d$ is said to be \emph{bounded} if 
we have $\sup_j \int |d\mu_j| < + \infty$,
and the sequence $ \{ \mu_j \}$ is  said to \emph{converge
vaguely} to a finite measure $\mu$ if  we have
$\int \pphi d\mu_j \to \int \pphi d\mu$
for every continuous, compactly supported 
function $\pphi$.
Every bounded sequence of measures $\{ \mu_j \}$
 has a vaguely convergent subsequence.

\subsection{}
We use  $\delta_\lam$ to denote  the Dirac measure at the point $\lambda$.
If $\Lam \sbt \R^d$ is a finite or countable set, then we denote
$\del_\Lam = \sum_{\lam \in \Lam} \del_\lam$.

By  a \emph{lattice} $L \subset \R^d$ we mean the image of $\Z^d$ under an
invertible linear map $T$. The determinant $\det(L)$ is equal to $|\det (T)|$.
The dual lattice $L^*$ is the set of all vectors $s$ such that $\dotprod{l}{s}
 \in \Z$, $l \in L$.
The measure $\del_L$ is a tempered distribution, whose Fourier transform
is (by Poisson's summation formula)
the measure $\ft{\del}_L = (\det L)^{-1} \sum_{s \in L^*} \del_{s}$.

We say that a set $\Lam \sbt \R^d$ is 
\emph{locally finite} if the set $\Lam \cap B$ is finite
for every open ball $B$. We say that $\Lam$ is 
\emph{periodic} if there exists a lattice $L$ such that
$\Lam + L = \Lam$. If $\Lam$ is both locally finite
and periodic, then it
is a union of finitely many translates of $L$.

% ================================================

\section{The Tur\'{a}n problem and its dual}
\label{s:turan}

\subsection{Admissible domains for the Tur\'{a}n problem}
An open set $U \sbt \R^d$ is said to have a
\emph{continuous boundary}
if for each point $a \in \partial{U}$
there exist an open ball $B$ centered at $a$,
an orthogonal linear map $\pphi: \R^d \to \R^d$
and a continuous function $\psi: \R^{d-1} \to \R$,
such that
\begin{equation}
U \cap B = \{
\pphi(x_1, \dots, x_d)  :
x_d < \psi(x_1, \dots, x_{d-1})\} \cap B.
\end{equation} 

This means that
locally near each boundary point, the set $U$
consists of those points lying to one side
of the graph of some continuous function.
This geometric condition is quite general
and is satisfied by most
domains of practical interest.

We now fix an open set  $U \sbt \R^d$ with the following properties:
\begin{enumerate-num}
\item
$U$ is an open set of finite measure;
\hfill (\refstepcounter{equation}\label{it:tuadmiti}\theequation)
\item
$0 \in U = -U$, that is, $U$ is  origin-symmetric 
and contains  the origin;
\hfill (\refstepcounter{equation}\label{it:tuadmitii}\theequation)
\item
$U$ has a continuous boundary.
\hfill (\refstepcounter{equation}\label{it:tuadmitiii}\theequation)
\end{enumerate-num}

Note that the set $U$ may be
unbounded, disconnected, or both.

\subsection{The Tur\'{a}n constant}
\label{ss:turan}

A function $f$  on $\R^d$ will be called
\emph{Tur\'{a}n admissible} (with respect to the set $U$) 
if $f$ is a bounded continuous
real-valued function  vanishing on $U^\cm$,
such that $\ft{f}$ is nonnegative and $f(0) = 1$.

We first observe that if $f$ is 
Tur\'{a}n admissible then $f \in L^1(\R^d)$, 
since $f$ is bounded and vanishes off the set $U$
of finite measure. In turn, since
$f$ is continuous
and $\ft{f}$ is nonnegative,
it follows that
$\ft{f}$ is a continuous function 
belonging to  $L^1(\R^d)$.
Moreover, we have
$\|f\|_\infty = \int \ft{f} =  f(0)=1$.
Finally, both $f$ and $\ft{f}$ are even functions.

\begin{definition}
\label{def:turanconst}
The \emph{Tur\'{a}n constant} $T(U)$ is the supremum
of $\int f$ over all  the
  Tur\'{a}n admissible functions $f$.
\end{definition}

We note that the Tur\'{a}n constant is finite,
and in fact, $T(U) \le m(U)$. This is due to
the fact that $\|f\|_{\infty} \le 1$
 for any Tur\'an admissible function $f$.

\subsection{Difference sets and convex domains}
\label{ss:diffandconvex}
As an example, suppose that $U$ contains a difference set $A-A$,
where $A \sbt \R^d$
is a  bounded open set.
In this case, the function
 $f = m(A)^{-1} \1_A \ast \1_{-A}$ is 
Tur\'{a}n admissible, and 
$\int f = m (A)$.  As a consequence, 
the lower estimate $T(U) \ge m(A)$ holds.

A special case of particular interest is when 
 $U \sbt \R^d$ is a \emph{convex} 
bounded  origin-symmetric  open set. In this case, $U$
can be realized as the difference set $U=A-A$
 where $A = \frac1{2}U$. Hence  the function
$f = m(A)^{-1} \1_A \ast \1_{-A}$ is 
Tur\'{a}n admissible, and 
$T(U) \ge m(A) = 2^{-d} m(U)$. 
We note that any bounded open convex set 
has a continuous boundary,
see \cite[Corollary 1.2.2.3]{Gri85}, 
so $U$  satisfies
\eqref{it:tuadmiti}, \eqref{it:tuadmitii}, \eqref{it:tuadmitiii}. 

If a convex bounded  origin-symmetric  open set $U \sbt \R^d$
satisfies $T(U)= 2^{-d} m(U)$, then $U$
is called a \emph{Tur\'an domain}, cf. \cite{KR03}.
It is not known whether there exists a convex
bounded origin-symmetric open set $U \sbt \R^d$
 which is not a Tur\'an domain, i.e.\ such that $T(U) > 2^{-d} m(U)$.

\subsection{The dual Tur\'{a}n constant}
\label{ss:dual-turan}

We say that a tempered distribution  $\al$ on $\R^d$ 
is admissible for the dual Tur\'{a}n problem,
if it is of the form
$\al = \del_0 + \be$, where $\be$ is a tempered distribution supported in
the closed set $U^\cm$, and moreover $\al$ is positive definite,
which means that $\ft{\al}$ is a positive measure.

In this case, 
we may  write $\ft{\al} =   \ft{\al} (\{0\})  \del_0 + \mu$,
where $\ft{\al} (\{0\}) $ is the mass of the
atom at the origin, and $\mu$ is a positive measure on $\R^d$.

\begin{definition}
The \emph{dual Tur\'{a}n constant} $T'(U)$ is the supremum
of $ \ft{\al} (\{0\})$ over all  the
tempered distributions $\al$
which are admissible for the dual Tur\'{a}n problem.
\end{definition}

We first observe that the dual
Tur\'{a}n constant $T'(U)$ is a
strictly positive number, and in fact,
$T'(U) \ge m(U)^{-1}$.
Indeed, the 
tempered distribution $\al$ given by
\begin{equation}
\label{eq:T.1.8.3}
\al = \del_0 + m(U)^{-1} \1_{U^\cm},
\quad
\ft{\al} = m(U)^{-1} \del_0
+ (1 - m(U)^{-1} \ft{\1}_U),
\end{equation}
is admissible for the 
dual Tur\'{a}n problem 
 and satisfies 
$ \ft{\al} (\{0\}) = m(U)^{-1}$.

\subsection{Weak linear duality in the Tur\'{a}n problem}
\label{sec:T1.3}
\label{ss:weak-turan}

The first result we obtain is the following
inequality involving  the
Tur\'{a}n constant $T(U)$ and its dual $T'(U)$.

\begin{thm}
\label{thm:T1.1}
Let $U \sbt \R^d$ be an open set satisfying 
\eqref{it:tuadmiti}, \eqref{it:tuadmitii}, \eqref{it:tuadmitiii}.
Then 
\begin{equation}
\label{eq:T1.1}
T(U) T'(U) \le 1.
\end{equation}
\end{thm}

Note that since $T(U)$ is strictly positive,
this shows that $T'(U)$ must be finite.

The inequality \eqref{eq:T1.1}  implies that any 
tempered distribution  $\al$ 
admissible
 for the dual Tur\'{a}n problem, yields the upper bound 
 $T(U) \le \ft{\al}(\{0\})^{-1}$
 for  the Tur\'{a}n constant. This principle is
usually referred to as  \emph{weak linear duality}.

The proof of \thmref{thm:T1.1} requires several observations.

\subsubsection{}
First we need the following lemma,
 based on \cite[Lemma 2.4]{MMO14}. 
Note that the assumption that $U$ has
a continuous boundary plays a crucial role
in the proof.

\begin{lem}
\label{lem:T3.7.2}
Let $f$ be Tur\'{a}n admissible for an open set $U$ satisfying \eqref{it:tuadmiti}, \eqref{it:tuadmitii}, \eqref{it:tuadmitiii}. For any  $\eps > 0$
there is a smooth real-valued function $g$ with compact
support contained in $U$, $g(0)=1$, such that  
 $\|\ft{f} - \ft{g}\|_1 < \eps$.
\end{lem}

Note that the approximating function $g$ is generally
not Tur\'{a}n admissible, since its 
Fourier transform $\ft{g}$ need not be
a nonnegative function.

\begin{proof}[Proof of  \lemref{lem:T3.7.2}]
Let $\psi$ be a smooth real-valued function with compact support,
$\psi(0)=1$, such that $\ft{\psi}$ is nonnegative. Then the
function 
$f_\del(x) := f(x)   \psi(\del x) $
is Tur\'{a}n admissible, has
compact support, and
$\|\ft{f}_\del - \ft{f}\|_1 \to 0$
as $\del \to 0$. 
Hence, with no loss of generality we may assume that $f$ has compact support.

The closed support of $f$ is thus a
compact set $K$ contained in the
closure of $U$. Since $U$ has
a continuous boundary,  for each point
$a \in K \cap \partial U$ there is
a  small 
 open ball $V(a)$ centered at $a$,
 and there is a unit vector $\tau(a)$,
such that $K \cap V(a) + \del \tau(a) \sbt U$
for any sufficiently small $\del > 0$.
By compactness we may choose
finitely many points
$a_1, \dots, a_n \in K \cap \partial U$
such that the open balls
$V_j := V(a_j)$, $1 \le j \le n$,
cover  $K \cap \partial U$.

If we denote $V_0:=U$, then $V_0, V_1, \dots, V_n$ forms an
open cover of $K$.
Let $\pphi_0, \dots, \pphi_n$
be a smooth partition of unity 
subordinate to this open cover,
that is, each $\pphi_j$ is a smooth real-valued function with
 compact support contained in 
$V_j$,  and $\sum_{j=0}^{n} \pphi_j(x) = 1$ on $K$.
Hence, if
we denote $f_j := f \cdot \pphi_j$ and $\tau_j := \tau(a_j)$, 
then the function
\begin{equation}
h_\del(x) = f_0(x) + \sum_{j=1}^{n} f_j(x - \del \tau_j)
\end{equation}
has compact support contained in  $U$, and
 $\|\ft{f} - \ft{h}_\del \|_1  \to 0$ as $\del \to 0$.
Moreover, since the origin belongs to the interior of $U$, 
we may assume that the open balls $V_1, \dots, V_n$ do not contain
the origin, which implies that $h_\del(0)=f(0)=1$.

It thus remains to set $g := (h_\del \ast \chi)(0)^{-1} (h_\del \ast \chi) $,
 where
 $\chi$ is a smooth nonnegative function supported on a
 sufficiently small neighborhood of the origin, with $\int \chi = 1$.
\end{proof}

\subsubsection{}
The next observation implies that if
a tempered distribution $\al$ is admissible 
for the dual Tur\'{a}n problem, 
then the measure 
$\ft{\al}$ is translation-bounded.

\begin{lem}
\label{lem:T6.9.1}
Let $\al$ be a tempered distribution on $\R^d$ such that
$\al = \del_0$ in some open neighborhood
$V$ of the origin, and 
$\ft{\al}$ is a positive measure.
Then the measure
$\ft{\al}$ is  translation-bounded.
Moreover, we have
$ \sup_t \ft\al(B + t) \le C(V)$,
where $B \sbt \R^d$
is the open unit ball 
and $C(V)$ is a constant which depends only on $V$.
\end{lem}

\begin{proof}
We choose and fix a smooth function  $\pphi$  with
compact support contained in $V$, and satisfying 
$\ft{\varphi}(-t) \ge  \1_B(t)$ for all $t \in \R^d$.
Then
\begin{equation}
 \ft{\al}(B+t)=\int \1_{B}(y-t)d\ft{\al}(y)\le 
 \int \ft{\varphi}(t-y)d\ft{\al}(y)
 = (\ft{\varphi}\ast \ft{\al})(t),
\end{equation}
due to the positivity of the measure $\ft{\al}$.
But note that
$\pphi \cdot \al = \pphi(0) \del_0$, 
which in turn implies that
$\ft{\varphi} \ast \ft{\al} = \pphi(0)$.
Hence the assertion holds with the
constant $C(V) = \pphi(0)$.
\end{proof}

\subsubsection{}
Let $f$ be any Tur\'{a}n  admissible function,
and let  $\al$ be any tempered distribution admissible 
for the dual Tur\'{a}n problem.
Then
$\ft{\al}$ is a positive, translation-bounded measure
(due to \lemref{lem:T6.9.1}),
while $\ft{f}$ is a nonnegative function in $L^1(\R^d)$.
It follows that the convolution $\ft{f} \ast \ft{\al}$ is a well-defined,
translation-bounded positive measure, which is also 
a locally integrable function. 

\begin{lem}
\label{lem:T3.1}
Let $f$ be a Tur\'{a}n admissible function,
and let $\al$ be an admissible tempered distribution  
 for the dual Tur\'{a}n problem.
Then   $\ft{f} \ast \ft{\al}  = 1$ a.e.
\end{lem}

This is obvious if $f$ is a Schwartz function 
whose closed support is contained in $U$, 
since in this case the product $f \cdot \al$
is well-defined and is equal to $\del_0$,
and therefore  the convolution $\ft{f} \ast \ft{\al}$,
being the Fourier transform 
of $f \cdot \al$, is 
 the constant function $1$.

However,
if  $f$ is only a continuous function and 
$\al$ is a tempered distribution, then generally
 the product $f \cdot \al$ \emph{does not} make sense.
Hence, to prove \lemref{lem:T3.1}
 in the general case,  we shall  use the 
  approximation result given in \lemref{lem:T3.7.2}.

\begin{proof}[Proof of  \lemref{lem:T3.1}]
By \lemref{lem:T3.7.2} there is a sequence
of smooth real-valued functions $g_j$ with compact
support contained in $U$, $g_j(0)=1$, such that 
$\ft{g}_j \to \ft{f}$ in $L^1(\R^d)$.
The extra smoothness and support properties
of the functions $g_j$ imply that 
$g_j  \cdot \al =  \del_0$,
and as a consequence,
$\ft{g}_j \ast \ft{\al}$ is the constant function $1$. 
Now we let $j \to \infty$. 
Let $\psi$ be a smooth function with compact support. 
Since $\ft{\al}$ is a
 translation-bounded  measure,  the convolution
 $\ft{\al} \ast \psi$ is a bounded function. 
 Since $\ft{g}_j \to \ft{f}$ in $L^1(\R^d)$,
 it follows that 
$ \ft{g}_j  \ast ( \ft{\al} \ast \psi) \to 
 \ft{f}  \ast (\ft{\al} \ast \psi)  $ pointwise.
  In turn, this implies that 
  $( \ft{g}_j  \ast  \ft{\al}) \ast \psi \to 
( \ft{f}  \ast \ft{\al}) \ast \psi  $ pointwise,
since the convolution is associative  (by Fubini's theorem).
But $\ft{g}_j \ast \ft{\al} = 1$, so we conclude that
$( \ft{f}  \ast \ft{\al}) \ast \psi  = \int \psi $.
Since this holds
for an arbitrary smooth function $\psi$ with compact support,
this shows that 
$\ft{f} \ast \ft{\al} = 1$ a.e.
\end{proof}

\subsubsection*{Remark}
\lemref{lem:T3.1} does not hold if we drop 
the assumption that $U$ has a continuous boundary.
An example constructed in \cite[Section 3]{Lev22} shows
that there is a bounded open set $U \sbt \R$
satisfying \eqref{it:tuadmiti} and \eqref{it:tuadmitii}, 
but not \eqref{it:tuadmitiii}, such that for certain
admissible  $f$ and $\al$,
the function $\ft{f} \ast \ft{\al}$ does not
coincide a.e.\ with any constant.

\subsubsection{}
Finally we can 
establish the weak  linear duality
inequality \eqref{eq:T1.1}.

\begin{proof}[Proof of \thmref{thm:T1.1}]
Due to the definitions of the 
Tur\'{a}n constant $T(U)$ and its dual $T'(U)$,
it suffices to verify that if $f$ is a 
Tur\'{a}n admissible function,
and if $\al$ is a tempered distribution admissible 
for the dual Tur\'{a}n problem, 
then $ \ft{\al} (\{0\}) \int f \le 1$.

 We may write 
$\ft{\al} =   \ft{\al} (\{0\})  \del_0 + \mu$,
where $\mu$ is a positive measure.
By \lemref{lem:T3.1},
\begin{equation}
\label{eq:T1.1.6}
1 = \ft{f} \ast \ft{\al} = \ft{\al} (\{0\}) \ft{f} + \ft{f} \ast \mu 
\ge  \ft{\al} (\{0\}) \ft{f} 
\quad \text{a.e.}
\end{equation}
Since $\ft{f} $ is a 
continuous function,  this implies that
the inequality
$ \ft{\al} (\{0\}) \ft{f}(x) \le 1 $  must in fact
hold for every $x \in \R^d$.
In particular, we have
$ \ft{\al} (\{0\}) \ft{f}(0)  \le 1$,
as required.
\end{proof}

\subsection{Strong linear duality in the Tur\'{a}n problem}
\label{sec:T1.4}\label{ss:strong-turan}

Our next goal is to show that
the inequality \eqref{eq:T1.1} is in fact an equality.

\begin{thm}
\label{thm:T1.3}
Let $U \sbt \R^d$ be an open set satisfying 
\eqref{it:tuadmiti}, \eqref{it:tuadmitii}, \eqref{it:tuadmitiii}.
Then 
\begin{equation}
\label{eq:T1.3}
T(U) T'(U) = 1.
\end{equation}
\end{thm}

This is usually referred to as \emph{strong linear duality}.
This principle also inspires the
idea of the proof which will be given next.

\subsubsection{}
\label{sec:T1.4.8}
Let $X$ be the linear space over $\R$
 consisting of all the bounded continuous
real-valued and even functions $f$
  vanishing on $U^\cm$, and
such that $\ft{f} \in L^1(\R^d)$.
We note that if $f \in X$ then also
$\ft{f}$ is a real-valued and even continuous 
function.

We consider  $L^1(\R^d) \times \R \times \R$ 
as a Banach space  over $\R$ (the functions in
the first component are taken real-valued) and let $K$
be the subset consisting of all triples
\begin{equation}
\label{eq:T1.4}
(\ft{f} - u, {\textstyle \int}  f - a, f(0) + b)
\end{equation}
where $f \in X$, $u$ is a nonnegative function in $L^1(\R^d)$, 
and $a,b$ are nonnegative scalars. It is easy to see that
the set $K$ is a convex cone.

\begin{lemma}
\label{lem:T1.1}
The equality \eqref{eq:T1.3} holds if and only if the triple
\begin{equation}
\label{eq:T1.5}
(0,T'(U)^{-1},1)
\end{equation}
belongs to the closure of $K$.
Moreover, if this is the case then
there exists a Tur\'{a}n admissible function $f$ with $\int f = T(U)$.
\end{lemma}

\begin{proof}
In one direction this is obvious:
if \eqref{eq:T1.3} holds then there is a sequence  of
Tur\'{a}n admissible functions $f_j$ such that 
$\int f_j \to T'(U)^{-1}$. Then
$f_j \in X$, $f_j(0) = 1$, and $\ft{f}_j$ is a
nonnegative function in $L^1(\R^d)$, hence
taking $f = f_j$, $u = \ft{f}_j$,
and $a=b=0$
in \eqref{eq:T1.4} yields the triple
$(0, \int f_j, 1)$ belonging to $K$ 
whose limit is \eqref{eq:T1.5}.

To prove the converse direction, we
 suppose that $(\ft{f}_j - u_j, \int f_j - a_j, f_j(0) + b_j)$ 
is a sequence in $K$ converging to \eqref{eq:T1.5}, that is,
\begin{equation}
\label{eq:T1.6}
\|\ft{f}_j - u_j\|_1 \to 0, \quad 
{\textstyle \int}  f_j - a_j \to T'(U)^{-1}, 
\quad f_j(0)+b_j \to 1.
\end{equation}

Since $u_j$ is nonnegative,
we have 
\begin{equation}
\label{eq:T7.7.2}
\|u_j\|_1 = \tint u_j
= f_j(0) - \tint (\ft{f}_j - u_j) 
 \le f_j(0) + b_j + 
\|\ft{f}_j - u_j\|_1.
\end{equation}
It follows from \eqref{eq:T1.6} that  the right hand side 
of \eqref{eq:T7.7.2} tends to $1$ as $j \to \infty$,
hence
$\limsup \|u_j\|_1 \le 1$. In particular,
$\{u_j\}$ is a bounded sequence in $L^1(\R^d)$,
so by passing to a subsequence, we may assume that 
$u_j$ converges vaguely
 to some finite measure $\mu$,
which ought to be a positive measure since
 $u_j$ are nonnegative functions.
 
 In turn, we have 
$\|\ft{f}_j\|_1 \le \|u_j\|_1  + 
\|\ft{f}_j - u_j\|_1$,  and so 
$\limsup \|\ft{f}_j\|_1 \le 1$.
Hence, 
 $\{\ft{f}_j\}$ is a bounded sequence in $L^1(\R^d)$.
 Moreover, since we have
$\ft{f}_j - u_j \to 0$ in $L^1(\R^d)$ and therefore
also vaguely, the sequence $\{\ft{f}_j\}$ 
 converges vaguely to the same measure $\mu$.

If we now set $f = \ft{\mu}$, then
$f$ is a bounded continuous function.
It follows from the  vague convergence 
that $f_j \to f$ in the sense of tempered distributions.
Hence the function $f$ is real-valued and even.
Since each $f_j$ vanishes on $U^\cm$, then
   $f$ must be   supported in the closure of $U$,
or equivalently, $f$  vanishes in the interior of $U^\cm$.
Since $U$ has a  continuous boundary, 
 the closed set $U^\cm$ is 
equal to the closure of its interior, 
hence by continuity the function $f$ 
must in fact vanish in the whole set $U^\cm$.
Furthermore, 
\begin{equation}
\label{eq:T1.8.5.1}
\|f\|_\infty = f(0) = \tint d\mu \le 
\limsup \tint u_j \le 1,
\end{equation}
and as a consequence,   $\|f\|_1 = \int_U |f| \le m(U)$,
so $f \in L^1(\R^d)$. 
It follows  that $\mu = \ft{f}$
is actually a nonnegative continuous  function in $L^1(\R^d)$.

We now claim that $\int f_j \to \int f$. Indeed, given
 $\eps > 0$ we choose a large ball $B$ 
such that $m(U \setminus B) < \eps$, and let
$\pphi$ be a Schwartz function such that
$0 \le \pphi \le 1$, and $\pphi = 1$ on $B$.
Then
\begin{equation}
\label{eq:T1.8.5.2}
\int f_j -  \int f =
\int (f_j - f) \cdot \pphi + 
\int (f_j - f) \cdot (1 - \pphi).
\end{equation}
Since  $\limsup \|f_j\|_\infty \le 
\limsup \|\ft{f}_j\|_1 \le 1$,
and due to \eqref{eq:T1.8.5.1},
for all sufficiently large $j$ 
the function $(f_j - f) \cdot (1 - \pphi)$
is bounded in modulus by an absolute 
constant $C$, and it vanishes off
the set  $U \setminus B$. Hence
the second integral on the right hand side
of \eqref{eq:T1.8.5.2} is bounded
in modulus  by
$C \cdot m(U \setminus B) < C \eps$. The first
integral on the right hand side 
of \eqref{eq:T1.8.5.2} tends to zero as $j \to \infty$,
since  $f_j \to f$ in the sense of tempered distributions.
This implies that $| \int f_j - \int f | < C \eps$ for all 
sufficiently large $j$. As this holds for an arbitrarily
small $\eps$, this shows that $\int f_j \to \int f$
and establishes our claim.

Since the scalars $a_j$ are nonnegative, 
we conclude   from \eqref{eq:T1.6} that

\begin{equation}
\label{eq:T1.7}
{\textstyle \int}  f =
 \lim_{j \to \infty} {\textstyle \int}  f_j  \ge \lim_{j \to \infty} 
( {\textstyle \int}  f_j - a_j ) =  T'(U)^{-1}.
\end{equation}

At this point we note that we still do not know that
$f$ is  Tur\'{a}n  admissible,
since we have not shown that $f(0)=1$.
However, it  follows from \eqref{eq:T1.8.5.1}
that $0 \le f(0) \le 1$. Moreover,
since $f(0) = \|f\|_\infty$, the value
$f(0)$ must be strictly positive,  for otherwise
this would imply that
$f=0$ which contradicts \eqref{eq:T1.7}.

We have thus shown that  $0 < f(0) \le 1 $.
The function $f(0)^{-1} f$ 
 is therefore  Tur\'{a}n  admissible,
and has integral $f(0)^{-1} \int f
 \ge f(0)^{-1}  T'(U)^{-1}$ due to \eqref{eq:T1.7}. 
On the other hand, we have  $f(0)^{-1} \int f
\le T(U) \le  T'(U)^{-1}$
by the definition of the 
Tur\'{a}n constant  $T(U)$
and the inequality  \eqref{eq:T1.1}.
This implies that actually $f(0) = 1$ and 
thus $f$ is a  Tur\'{a}n  admissible function,
and 
$\int f = T(U) =  T'(U)^{-1}$. In particular,
the equality \eqref{eq:T1.3} holds.
This completes the proof of \lemref{lem:T1.1}.
\end{proof}

\subsubsection{}
\label{sec:LP2.1}
We now continue to the proof of the strong linear duality equality \eqref{eq:T1.3}.

\begin{proof}[Proof of \thmref{thm:T1.3}]
In view of \lemref{lem:T1.1}, 
in order to prove that  the equality \eqref{eq:T1.3} 
holds, it suffices to show
that the triple \eqref{eq:T1.5} must belong to the closure of $K$.
Suppose to the contrary that this is not the case. Since $K$ is convex, then
 by the Hahn-Banach separation theorem 
 (see e.g.\ \cite[Theorem 3.4]{Rud91})
there exists a continuous linear functional on the  
space   $L^1(\R^d) \times \R \times \R$ 
which separates  the closure of $K$ from the triple
\eqref{eq:T1.5}. This means that  there exists an element 
 $(g,p,q)$  of  the space $L^\infty(\R^d) \times \R \times \R$
  (again the functions in
the first component are taken real-valued),  and there is 
 a real scalar $c$, such that the inequality
\begin{equation}
\label{eq:T1.9}
{\textstyle \int} (\ft{f} - u) g + p ( {\textstyle \int} f  - a) - q ( f(0) + b)  \le c
\end{equation}
holds for every $f \in X$, every  nonnegative function $u \in L^1(\R^d)$, and
every nonnegative scalars $a, b$, while at the same time we have
\begin{equation}
\label{eq:T1.8}
 p T'(U)^{-1} - q > c.
\end{equation}
(The left hand side of \eqref{eq:T1.8}
is the action of $(g,p,q)$ on the triple
\eqref{eq:T1.5}.)

We first observe that the function 
$g$ must be nonnegative a.e. Indeed, if $g(t) < 0$
on some set $E$ of positive and finite
measure, then taking $f=0$, $a=b = 0$ and
$u  = \lam  \cdot \1_E$
would violate \eqref{eq:T1.9} for a sufficiently large positive scalar $\lam$.
Similarly, $p$ must be a nonnegative scalar, for otherwise
taking $f=0$, $u = 0$ and $b=0$ would violate \eqref{eq:T1.9} for a sufficiently large 
positive scalar $a$. In the same way  also
$q$ must be a nonnegative scalar. 
Thus $g$ is a nonnegative 
function in $L^\infty(\R^d)$,  and $p,q$ are nonnegative scalars.

If we now set $u = 0$ and $a=b=0$ in \eqref{eq:T1.9} then we obtain
\begin{equation}
\label{eq:T2.1}
{\textstyle \int} \ft{f}  g + p  {\textstyle \int} f   - q f(0)  \le c
\end{equation}
for every $f \in X$. Since $X$ is a linear space, 
it follows that the left hand side can never be nonzero
(for otherwise it could be made positive and arbitrarily large).
So we may assume that $c = 0$ and that the
inequality \eqref{eq:T2.1} is in fact an equality.

By replacing $g(t)$ with $\tfrac1{2} (g(t) + g(-t))$, we may
also assume that $g$ is a nonnegative and even function
in $L^\infty(\R^d)$.

If $f \in X$ is a smooth function with compact support
contained in $U$, 
then the  equality in  \eqref{eq:T2.1} may be written
in the distributional sense as
\begin{equation}
\label{eq:T2.2}
 ( \ft{g}  + p  - q  \del_0) (f) = c = 0.
\end{equation}
Moreover, the tempered distribution
$ \ft{g}  + p  - q \del_0$ is real and even, 
therefore the fact that the equality
\eqref{eq:T2.2} holds for all smooth functions
 $f \in X$  with compact support
contained in $U$,   implies that
$ \ft{g}  + p  - q \del_0$ vanishes in $U$.

Recall now that $p,q$ are nonnegative scalars. We claim
that in fact they are both strictly positive. 
Indeed, using \eqref{eq:T1.8} and recalling that  $c=0$,
it follows that $p>0$.  In turn, this implies that
  $g + p \del_0$ is a nonzero positive measure,
  whose Fourier transform satisfies
$ \ft{g}  + p = q \del_0$ in $U$. 
But since $ \ft{g} + p$ cannot vanish in any neighborhood
  of the origin, we conclude  that  also $q>0$.

Finally, define $\al = q^{-1} ( \ft{g}  + p )$. Then $\al = \del_0$ in $U$,
and $\ft{\al}$ is a positive measure. Hence $\al$ is a
tempered distribution admissible for the dual Tur\'{a}n problem.
Furthermore, the Fourier transform of $\al$ is given by
 $\ft{\al} = q^{-1}( p  \del_0 + g)$, so that
 the measure $\ft{\al}$ has an
 atom at the origin of mass $\ft{\al}(\{0\})
 = p/q$. However due to
\eqref{eq:T1.8} and recalling that  $c=0$,
we conclude that $\ft{\al}(\{0\}) > T'(U)$,
 which gives us the desired
contradiction.
\end{proof}

\subsubsection{Remark}
We note an interesting consequence of the last proof. Suppose that
we consider the following smaller class of admissible tempered distributions 
$\al$ on $\R^d$.
We again require that 
$\al = \del_0 + \be$, where $\be$ is a tempered distribution supported in
the closed set $U^\cm$, but in addition we require that $\ft{\al}$
is of the form  $\ft{\al} = \ft{\al}(\{0\}) \del_0 + g$, where $g$
is a nonnegative even function in $L^\infty(\R^d)$. Then
the supremum of $\ft{\al}(\{0\}) $ over this smaller class of admissible 
$\al$'s still gives us the dual Tur\'{a}n constant $T'(U)$.

\subsection{Existence of extremizers for the Tur\'{a}n problem and its dual}
\label{ss:existence-turan}
We say that a Tur\'{a}n admissible function $f$
is extremal if it satisfies $\int f = T(U)$.
Similarly,  a tempered distribution  $\al$ which
is admissible for the dual Tur\'{a}n problem
will be called extremal
if we have  $ \ft{\al} (\{0\})  = T'(U)$.

Our next goal is to establish the existence of
extremizers for both the Tur\'{a}n problem and
its dual. This implies that in the
definition of the constants $T(U)$ and $T'(U)$,
the supremum is in fact a maximum.

The existence of an extremal function for the Tur\'{a}n 
problem was proved in \cite[Corollary 19]{BRR24}
in the more general context of locally compact
abelian groups.

\begin{thm}
\label{thm:T8.1}
Let $U \sbt \R^d$ be an open set satisfying 
\eqref{it:tuadmiti}, \eqref{it:tuadmitii}, \eqref{it:tuadmitiii}.
Then,
\begin{enumerate-roman}
\item \label{it:tex.1}
The Tur\'{a}n problem admits at least
one extremal function $f$;
\item \label{it:tex.2}
The dual Tur\'{a}n problem admits at least
one extremal  tempered distribution $\al$.
\end{enumerate-roman}
\end{thm}

\begin{proof}
Since the equality \eqref{eq:T1.3}
of \thmref{thm:T1.3} is now
proved, then part \ref{it:tex.1} follows
 as a consequence of
\lemref{lem:T1.1}.
It thus remains to prove part \ref{it:tex.2}.

Let  $\{\al_j\}$ be a sequence  of
tempered distributions  which are 
admissible for the dual
Tur\'{a}n problem, and such that
$\ft{\al}_j( \{0 \})  \to T'(U)$.
It follows from \lemref{lem:T6.9.1}
that $\{\ft{\al}_j\}$ is a 
uniformly translation-bounded
 sequence of positive measures.
 As a consequence, by passing to a subsequence 
 we may assume that $\ft{\al}_j$ converges
 vaguely  to some translation-bounded positive
 measure, which we may denote as $\ft{\al}$
 for some tempered distribution  $\al$.
 The  uniform   translation-boundedness and the
 vague convergence imply that the sequence
 $\ft{\al}_j$ converges to $\ft{\al}$ also
in the sense of  tempered distributions.
As a consequence, since $\al_j = \del_0$ in $U$ for all $j$, then also
 $\al = \del_0$ in $U$, hence  $\al$ is a
 tempered distribution admissible 
for the dual Tur\'{a}n problem.
Moreover,  since the measure $\ft{\al}_j$ 
is  positive and has mass 
$\ft{\al}_j( \{0 \}) $
at the origin, then the vague limit $\ft{\al}$
also has an atom at the origin, of mass at least 
$\lim \ft{\al}_j( \{0 \})  = T'(U)$.  Hence
$\al$ is extremal for the 
dual Tur\'{a}n problem.
\end{proof}

\subsection{Relation between extremizers
 for the Tur\'{a}n problem and its dual}
\label{ss:properties-turan}
Let $f$ and $\al$ be 
extremizers for the Tur\'{a}n problem and 
its dual, respectively. 
This means that
 $f$ is a Tur\'{a}n admissible function,  with $\int f = T(U)$,
 and  that $\al$ is a tempered distribution admissible 
 for the dual Tur\'{a}n problem,
 with $\ft{\al}(\{0\}) = T'(U)$.
Note that $\ft{f}$ is a continuous function 
and   $\ft{\al}$ is a measure, so the product
$\ft{f} \cdot \ft{\al}$  is well-defined.

\begin{thm}
\label{thm:T3.7}
Let $U \sbt \R^d$ be an open set satisfying 
\eqref{it:tuadmiti}, \eqref{it:tuadmitii}, \eqref{it:tuadmitiii}.
If $f$ and $\al$ are any two extremals
for the Tur\'{a}n problem and
 its dual, respectively, then the  
measure $\ft{\al}$ is supported on the closed set
 $\{t : \ft{f}(t)=0\} \cup \{0\}$,
  and as  a consequence,  
$\ft{f} \cdot \ft{\al} = \del_0$.
\end{thm}

\begin{proof}
We can write
  $\ft{\al} = T'(U) \del_0 + \mu$,
where $\mu$ is a positive measure.
We claim that
$\mu$ is supported on the set
$\{t : \ft{f}(t)=0\}$.  
Suppose to the contrary that
 this is not the case, then
 there is a point $a$ in
 the closed support of $\mu$ such that $\ft{f}(a)>0$.
Let $\psi$ be a smooth function 
with compact support, satisfying
$0 \le \psi \le \ft{f}$  and such that
$\psi(t) > 0$ in some open neighborhood $V$ of  
the point $a$. 
Since $\ft{f}$ is an even function, we may assume
that $\psi$ is even as well. 
Hence $\psi \ast \mu$ is 
a  nonnegative continuous function 
with
$(\psi \ast \mu)(0)  = \int \psi(t) d\mu(t) > 0$,
where the strict inequality holds since the positive
measure $\mu$ has nonzero mass in $V$ while
$\psi(t) > 0$ for $t \in V$.
By \lemref{lem:T3.1}, we have
\begin{equation}
\label{eq:T5.1}
1 = \ft{f} \ast \ft{\al} = T'(U) \ft{f} + \ft{f} \ast \mu \quad \text{a.e.}
\end{equation}
The right hand side of \eqref{eq:T5.1} is therefore
a.e.\ not less than
$ T'(U) \ft{f} +   \psi \ast  \mu$,
which is a  nonnegative continuous 
function whose value at the origin is 
 strictly  greater than $T'(U) \ft{f}(0) = T'(U) T(U) = 1$,
which gives us a contradiction. This shows that the measure
$\mu$ must indeed be supported on the set
$\{t : \ft{f}(t)=0\}$. In turn, this implies that
 $\ft{\al}$ is supported on the set
 $\{t : \ft{f}(t)=0\} \cup \{0\}$
and that $\ft{f} \cdot \ft{\al} = \del_0$.
\end{proof}

\subsubsection*{Remark}
Note that we \emph{do not} say that $f \ast \al = 1$, since 
generally if  $f$ is a continuous function and 
$\al$ is a tempered distribution, then
the convolution $f \ast \al$ \emph{does not} make sense.

% =========================================================

\section{The Delsarte problem and its dual}
\label{s:delsarte}

\subsection{Admissible domains for the Delsarte problem}

We now consider a wider class of domains.
We fix an open set  $U \sbt \R^d$ satisfying the following 
two properties:
\begin{enumerate-num}
\item
$U$ is an open set of finite measure;
\hfill (\refstepcounter{equation}\label{it:dsadmiti}\theequation)
\item
$0 \in U = -U$, that is, $U$ is  origin-symmetric 
and contains  the origin;
\hfill (\refstepcounter{equation}\label{it:dsadmitii}\theequation)
\end{enumerate-num}
Later on, we will also assume that:
\begin{enumerate-num}
\setcounter{enumi}{2}
\item
The closed set $U^\cm$ is equal to the closure of its interior.
\hfill (\refstepcounter{equation}\label{it:dsadmitiii}\theequation)
\end{enumerate-num}

The first two conditions \eqref{it:dsadmiti}, \eqref{it:dsadmitii}
coincide with \eqref{it:tuadmiti}, \eqref{it:tuadmitii}.
The third condition  \eqref{it:dsadmitiii} will not be assumed from
the beginning, since it is not needed in the proof of the
weak linear duality. We will impose the condition 
\eqref{it:dsadmitiii} later on, when we establish the strong linear duality
and the results which follow it.

The condition  \eqref{it:dsadmitiii} 
is a weaker requirement than \eqref{it:tuadmitiii}.
Hence, we consider a more general class of  admissible domains
for the Delsarte problem. 
The  condition  \eqref{it:dsadmitiii} is also considered in
\cite{BRR24} where it is called ``boundary coherence''.

Note again, that the set $U$ may be unbounded, disconnected, or both.

\subsection{The Delsarte constant}
\label{ss:delsarte}

We  say that a function  $f$
on $\R^d$ is \emph{Delsarte admissible}  if $f$ is a continuous 
real-valued function  in $L^1(\R^d)$ satisfying 
the conditions
$f(0)=1$, $f(t) \le 0$ for $t \in U^\cm$,
and   $\ft{f}$ is a nonnegative  function.

If $f$  is Delsarte admissible
then both $f$ and $\ft{f}$ are even functions,
and   $\ft{f}$ is a continuous function 
belonging to  $L^1(\R^d)$. Moreover,
we have
$\|f\|_\infty = \int \ft{f} = f(0)=1$.

\begin{definition}
The \emph{Delsarte constant} 
 $D(U)$ is the supremum
of $\int f$ over all  the Delsarte
admissible functions $f$.
\end{definition}

The Delsarte constant $D(U)$ is finite,
and satisfies $D(U) \le m(U)$.

We observe that $D(U) \ge T(U)$, that is, the Delsarte 
constant is at least as large as the Tur\'{a}n constant,
since the supremum is taken over a larger class of 
admissible functions. Indeed, 
in the Tur\'{a}n problem we require that $f$
vanishes on $U^\cm$, 
while in the Delsarte problem $f$ is only
required to be nonpositive in $U^\cm$. 

In particular (recall \secref{ss:diffandconvex})
 this implies that 
if $U$ contains a difference set $A-A$,
where $A \sbt \R^d$
is a  bounded open set, then $D(U) \ge m(A)$.

\subsection{The dual Delsarte constant}
\label{ss:dual-delsarte}

We say that a tempered distribution  $\al$ on $\R^d$ 
is admissible for the dual Delsarte problem,
if $\al = \del_0 + \be$, where $\be$ is a
positive measure supported in the closed set $U^\cm$,
 and   $\al$ is positive definite,
which means that $\ft{\al}$ is a positive measure.
In this case, 
we may as before write $\ft{\al} =   \ft{\al} (\{0\})  \del_0 + \mu$,
where $\ft{\al} (\{0\}) $ is the mass of the
atom at the origin, and $\mu$ is a positive measure.

\begin{definition}
The \emph{dual Delsarte constant} $D'(U)$ is the supremum
of $ \ft{\al} (\{0\})$ over all  the
tempered distributions $\al$
which are admissible for the dual Delsarte problem.
\end{definition}

Similarly, we note that $D'(U) \le T'(U)$, since 
the supremum is taken over a smaller class of admissible 
tempered distributions $\al$. Indeed, in the dual Delsarte problem
we require $\be$ to be a positive measure, while 
in the dual Tur\'{a}n problem, 
 $\be$ is merely a tempered distribution 
 which need not be a measure.

We observe that
the tempered distribution $\al$ given by
\eqref{eq:T.1.8.3}
is admissible for the dual Delsarte problem,
hence
$D'(U) \ge m(U)^{-1}$.
In particular, the dual Delsarte
constant is strictly positive.

It follows from
\lemref{lem:T6.9.1} that
if $\al$ is a tempered distribution admissible 
for the dual Delsarte problem, then 
 $\ft{\al}$ is a translation-bounded measure.

\subsection{Weak linear duality in the Delsarte problem}
\label{sec:D1.3}\label{ss:weak-delsarte}
We now turn to  prove  the inequality that establishes
the weak linear duality in the Delsarte problem.
Note that this result does not require
the condition  \eqref{it:dsadmitiii}.

\begin{thm}
\label{thm:D1.1}
Let $U \sbt \R^d$ be an open set 
satisfying \eqref{it:dsadmiti}, \eqref{it:dsadmitii}. Then 
\begin{equation}
\label{eq:D1.1}
D(U) D'(U) \le 1.
\end{equation}
\end{thm}

In particular, this shows that the dual Delsarte 
constant $D'(U)$ is finite.

Moreover, the inequality \eqref{eq:D1.1} gives us
as before, that any 
tempered distribution  $\al$ 
admissible
 for the dual Delsarte problem yields the upper
 bound 
 $D(U) \le \ft{\al}(\{0\})^{-1}$.

\begin{proof}[Proof of \thmref{thm:D1.1}]
Let $f$ be any Delsarte admissible function,
and let $\al$ be any tempered distribution admissible 
for the dual Delsarte problem.
Note that although $f$ is not necessarily a smooth function,
the product $f \cdot \al$ is a well-defined signed measure,
since $f$ is a continuous function and $\al$ is a measure.
Moreover, since $\ft{f}$ is a function in $L^1(\R^d)$ and 
  $\ft{\al}$ is  a translation-bounded
  measure, then by \lemref{lemB1}
  the signed measure
$f \cdot \al$ is a tempered distribution whose
Fourier transform is $\ft{f} \ast \ft{\al}$,
which  is a well-defined, positive translation-bounded measure, 
and which is also 
a locally integrable function. 

Fix a nonnegative Schwartz function  $\pphi$ with
$\int \pphi = 1$, such that $\ft{\pphi}$ is nonnegative
and has compact support.
Let $\pphi_\eps(t) = \eps^{-d} \pphi(t/\eps)$,
then 
$\ft{\pphi}_\eps(x)  = \ft{\pphi}(\eps x)$.
The Fourier transform $\ft{\al}$ is of the form
$\ft{\al} =   \ft{\al} (\{0\})  \del_0 + \mu$,
where $\mu$ is a positive measure, hence
\begin{equation}
\label{eq:D3.1}
\ft{f} \ast \ft{\al} 
= \ft{\al}(\{0\}) \ft{f} + \ft{f} \ast  \mu.
\end{equation}
Since $\pphi_\eps$ is nonnegative, 
and $\ft{f} \ast  \mu $ is a positive measure, 
this implies that 
\begin{equation}
\label{eq:D3.2}
(\ft{f} \ast \ft{\al} )(\pphi_\eps ) 
 \ge \ft{\al}(\{0\}) \int  \ft{f}(t) \pphi_\eps(t)dt   \to
 \ft{\al}(\{0\}) \ft{f}(0) = \ft{\al}(\{0\}) \int f
\end{equation}
as $\eps \to 0$. On the other hand,
$\al = \del_0 + \be$ where $\be$ is a
positive measure, hence
\begin{equation}
\label{eq:D3.3}
(\ft{f} \ast \ft{\al} )(\pphi_\eps ) =
(f \cdot \al)(\ft{\pphi}_\eps ) =
f(0)\ft{\pphi}_\eps (0) + \int \ft{\pphi}_\eps(x) f(x) d \beta(x)
\end{equation}
(the last equality holds and the  integral is well-defined since 
$\ft{\pphi}_\eps$ has compact support).
Since $\be$ is a positive measure supported on
$U^\cm$, $f$ is nonpositive on $U^\cm$ and
$\ft{\pphi}_\eps$ is nonnegative everywhere, it follows
that the integral in \eqref{eq:D3.3} is a nonpositive scalar.
Hence
\begin{equation}
\label{eq:D3.5}
(\ft{f} \ast \ft{\al} )(\pphi_\eps ) \le  f(0)\ft{\pphi}_\eps (0) = 1.
\end{equation}
Combining  \eqref{eq:D3.2}, \eqref{eq:D3.5}
 yields the inequality $ \ft{\al}(\{0\}) \int f \le 1$,
which proves \eqref{eq:D1.1}.
\end{proof}

\subsection{Strong linear duality in the Delsarte problem}
\label{ss:strong-delsarte}
Next we show that in fact we have an equality
in \eqref{eq:D1.1}, which establishes 
strong linear duality.

\begin{thm}
\label{thm:D1.3}
Let $U \sbt \R^d$ be an open set 
satisfying \eqref{it:dsadmiti}, \eqref{it:dsadmitii}, \eqref{it:dsadmitiii}. Then 
\begin{equation}
\label{eq:D1.3}
D(U) D'(U) = 1.
\end{equation}
\end{thm}

This result was proved under broader assumptions in \cite[Section 3]{CLS22}.
Here we give a different presentation following similar lines to our proof of the
strong linear duality for the Tur\'{a}n problem
(\thmref{thm:T1.3}).

\subsubsection{}
Let $Y$ be the linear space over $\R$
 consisting of all real-valued and even continuous 
functions $f$ such
that both $f$ and $\ft{f}$ are in $L^1(\R^d)$. Note that if $f \in Y$ then 
also $\ft{f}$ is a real-valued and even continuous 
function.

We now  consider $L^1(\R^d) \times L^1(U^\cm) \times \R \times \R$ 
as a Banach space  over $\R$ (the functions in
the first and second components are taken real-valued) and let $K$
be the set of all quadruples
\begin{equation}
\label{eq:D1.4}
(\ft{f} - u,  - f|_{U^\cm} - v, \tint f - a, f(0) + b)
\end{equation}
where $f \in Y$, $u$ is a nonnegative function in $L^1(\R^d)$, 
$v$ is a nonnegative function in $L^1(U^\cm)$, 
and $a,b$ are nonnegative scalars. Then the set $K$ is a convex cone.

\begin{lemma}
\label{lem:D1.1}
The equality \eqref{eq:D1.3} holds if and only if the quadruple
\begin{equation}
\label{eq:D1.5}
(0,0,D'(U)^{-1},1)
\end{equation}
belongs to the closure of $K$.
Moreover, if this is the case then
there exists a Delsarte  admissible function $f$ with
  $\int f = D(U)$.
\end{lemma}

\begin{proof}
Again one direction is obvious:
 if \eqref{eq:D1.3} holds then there is a sequence  of
Delsarte admissible functions $f_j$ such that $\int f_j \to D'(U)^{-1}$.
Then $f = f_j$ is a function in $Y$, $u = \ft{f}_j$ is a nonnegative
function belonging to $L^1(\R^d)$, $v = - f|_{U^\cm}$  is a nonnegative
function in $L^1(U^\cm)$, so together with $a=b=0$,
the quadruple \eqref{eq:D1.4} becomes
$(0, 0, \int f_j, 1)$ which belongs to $K$ and
whose limit is \eqref{eq:D1.5}.

We now must prove also the converse direction.
Assume that 
\begin{equation}
\label{eq:D1.6.5}
(\ft{f}_j - u_j, - f_j|_{U^\cm} - v_j,  \tint f_j - a_j, f_j(0)+b_j)
\end{equation}
 is a sequence
in $K$ converging to \eqref{eq:D1.5}, that is,
\begin{equation}
\label{eq:D1.6}
\tint_{\R^d} |\ft{f}_j - u_j| \to 0, \quad 
\tint_{U^\cm} |f_j + v_j| \to 0, \quad
 {\textstyle \int}  f_j - a_j \to D'(U)^{-1}, \quad f_j(0)+b_j \to 1.
\end{equation}

In  the same way as in 
the proof \lemref{lem:T1.1}, we can conclude that $\{u_j\}$ and
$\{\ft{f}_j\}$  are two bounded sequences
in $L^1(\R^d)$, with  $\limsup \|u_j\|_1$ and
$\limsup \|\ft{f}_j\|_1$ both not exceeding $1$,
and after passing to a subsequence, 
these two sequences converge vaguely to a common
vague limit which is a finite  positive measure $\mu$.

We now  show that also 
$\{f_j\}$ is a  bounded sequence in $L^1(\R^d)$. Indeed,
we have
\begin{equation}
\label{eq:D2.3.1}
\|f_j\|_1 = 
\int_{\R^d} (-f_j) + \int_{U}(f_j + |f_j|)
+ \int_{U^\cm}(f_j + |f_j|).
\end{equation}
The first integral $\int  (-f_j)$ does not exceed 
 $ a_j  - \int  f_j$ which tends to a limit
by \eqref{eq:D1.6}, hence this
integral remains bounded from above.
The second integral $\int_{U}(f_j + |f_j|)$
does not exceed $2 m(U) \|f_j\|_\infty$
which is bounded due to the
inequality 
$\|f_j\|_\infty \le \|\ft{f}_j\|_1$
and the fact that 
$\{\ft{f}_j\}$ is a bounded sequence in $L^1(\R^d)$.
To estimate the third
integral $\int_{U^\cm}(f_j + |f_j|)$
we observe that on the set $U^\cm$ we have
$f_j \le -v_j + |f_j  + v_j|$ as well as 
the inequality
 $|f_j| \le v_j + |f_j  + v_j|$ 
(since $v_j$ is a nonnegative function).
Hence $f_j + |f_j| \le 2|f_j  + v_j|$ and so 
the integral $\int_{U^\cm}(f_j + |f_j|)$
does not exceed
$2 \tint_{U^\cm} |f_j + v_j|$, which tends to zero
by \eqref{eq:D1.6} and in particular remains bounded.
Thus  the right hand side of
\eqref{eq:D2.3.1} is bounded from above,
and $\{f_j\}$ is a  bounded sequence in $L^1(\R^d)$.

As a consequence, again 
by passing to a  subsequence we may assume that
$f_j$ converges vaguely to a finite  signed measure $\nu$.
The vague convergence implies that both $f_j \to \nu$
and $\ft{f}_j \to \mu$
 in the sense of tempered distributions.
Hence
 we must have $\ft{\nu} = \mu$, which implies that in fact
both $\nu$ and $\mu$ are real-valued and even continuous 
functions belonging to $L^1(\R^d)$. We thus denote
$f = \nu$ and note that $\ft{f} = \mu$ is a nonnegative function.

We now claim that the function
$f$ is nonpositive in $U^\cm$.
To see this, let $\psi$ be a smooth nonnegative function 
with compact support contained in the interior of $U^\cm$. Then
\begin{equation}
\label{eq:D1.9.1}
 -\int_{\R^d}  f_j \psi =  \int_{U^\cm} v_j \psi -  \int_{U^\cm} ( f_j +  v_j ) \psi.
\end{equation}
The first integral on the right hand side is nonnegative, 
while the second integral tends to zero as $j \to \infty$ due to 
\eqref{eq:D1.6}. Hence 
using the vague convergence we obtain
\begin{equation}
\label{eq:D1.9.2}
 -\tint f \psi =  \lim_{j \to \infty}( - \tint f_j \psi )\ge 0.
\end{equation}
As this holds for an arbitrary
smooth nonnegative function $\psi$ 
with compact support contained in the interior of $U^\cm$,
and since $f$ is a continuous function,
this  implies that $f$  is nonpositive in the interior of $U^\cm$.
Since we have assumed that the closed set $U^\cm$ is 
equal to the closure of its interior, 
it follows again by the continuity of $f$ that 
$f$ must be nonpositive in the whole set $U^\cm$.

Next we claim that
$\int f \ge D'(U)^{-1}$. To prove this, 
let $\eps > 0$ be given.
We choose a large ball $B$ 
such that   $m(U \setminus B) < \eps$,
and also  
$\int_{B^\cm} |f| < \eps$. Let
$\pphi$ be a Schwartz function such that
$0 \le \pphi \le 1$, and $\pphi = 1$ on $B$.
We have
\begin{equation}
\label{eq:D6.1}
 \int ( 1 - \pphi ) \cdot  f_j = 
 \int_{U} (1 - \pphi) \cdot f_j
 + \int_{U^\cm} (1 - \pphi)\cdot ( f_j + v_j)
- \int_{U^\cm} (1 - \pphi)\cdot v_j.
\end{equation}
Since  $\limsup \|f_j\|_\infty \le 
\limsup \|\ft{f}_j\|_1 \le 1$,
the function $ (1 - \pphi) \cdot  f_j $
is bounded in modulus by an absolute
constant $C$ for all sufficiently large $j$,
and it vanishes on $B$. Hence
the first integral on the right hand side
of \eqref{eq:D6.1} does not exceed
$C \cdot m(U \setminus B) < C \eps$.
The second integral tends to zero as $j \to \infty$ 
due  to  \eqref{eq:D1.6}, while the 
third integral is nonnegative. 
This shows that
$\limsup \int ( 1 - \pphi) \cdot  f_j  < C \eps$.
Next, we have $f_j \to f$ 
 in the sense of tempered distributions, hence
\begin{equation}
\label{eq:D6.7.3}
\int f \cdot \pphi
= \lim_{j \to \infty} \int f_j  \cdot \pphi 
= \lim_{j \to \infty} \Big\{
(\tint f_j - a_j) + a_j -
 \tint ( 1 - \pphi) \cdot f_j  \Big\}.
\end{equation}
Since  $\tint f_j - a_j \to  D'(U)^{-1}$,
$a_j$ is a nonnegative scalar, and
$\limsup \int ( 1 - \pphi) \cdot  f_j  < C \eps$,
this implies that
\begin{equation}
\label{eq:D6.7.9}
\int f \cdot \pphi >  D'(U)^{-1} - C \eps.
\end{equation}
On the other hand,  since the function
 $ 1 - \pphi$ vanishes on $B$, we have
\begin{equation}
\label{eq:D6.7.4}
\int f \cdot \pphi = \int f - \int_{B^\cm} (1- \pphi ) \cdot f 
\le \int f + \int_{B^\cm} |f| < \int f + \eps.
\end{equation}
Hence \eqref{eq:D6.7.9} and \eqref{eq:D6.7.4}
yield that
$ \int f >   D'(U)^{-1} - (C+1) \eps$.
As this holds for any $\eps > 0$,
this shows that indeed 
$ \int f  \ge   D'(U)^{-1}$, and
thus our claim is established.

Finally, we  show that 
$f$ is a Delsarte admissible function, that is,
we need  to establish that $f(0)=1$.
In  the same way as in 
the proof \lemref{lem:T1.1}, we can show
at the first step that  we have $0 < f(0) \le 1 $
(using  the fact that $\limsup \int u_j \le 1$).
Hence  $f(0)^{-1} f$  is a Delsarte
admissible function, whose integral is not less than
$ f(0)^{-1}  D'(U)^{-1}$.
But on the other hand,
 $f(0)^{-1} \int f
 \le D(U) \le  D'(U)^{-1}$ due to the definition
 of the Delsarte constant $D(U)$ and the inequality
\eqref{eq:D1.1}. Hence $f(0) = 1$ and  $f$ is Delsarte
admissible, and moreover, 
$\int f = D(U) =  D'(U)^{-1}$. In particular, we conclude
that the equality \eqref{eq:D1.3} holds.
This completes the proof of \lemref{lem:D1.1}.
\end{proof}

\subsubsection{}
We continue to the proof of the strong linear duality equality \eqref{eq:D1.3}.

\begin{proof}[Proof of \thmref{thm:D1.3}]
Due to \lemref{lem:D1.1},
in order to prove  the equality \eqref{eq:D1.3}  it suffices to show
that the quadruple \eqref{eq:D1.5} must belong to the closure of $K$.
Suppose to the contrary that this is not the case. Since $K$ is convex, then
 the Hahn-Banach separation theorem 
  (see again \cite[Theorem 3.4]{Rud91}) 
  yields a continuous linear functional on the  
space   $L^1(\R^d) \times L^1(U^\cm) \times \R \times \R$ 
which separates  the closure of $K$ from the quadruple
\eqref{eq:D1.5}, that is, there exists an element 
 $(g,h,p,q)$  of  the space $L^\infty(\R^d)  \times L^\infty(U^\cm) \times \R \times \R$
  (again the functions in
the first and second components are taken real-valued),  and there is 
 a real scalar $c$, such that the inequality
\begin{equation}
\label{eq:D1.9}
\tint_{\R^d} (\ft{f} - u) g
 + \tint_{U^\cm} (- f - v) h 
 + p ( \tint f  - a) - q (f(0)+b)  \le c
\end{equation}
holds for every $f \in Y$, every  nonnegative function
$u\in L^1(\R^d)$, 
every nonnegative function $v \in L^1(U^\cm)$, 
and for every 
 nonnegative scalars $a,b$, while at the same time 
\begin{equation}
\label{eq:D1.8}
 p D'(U)^{-1} - q > c.
\end{equation}
(The left hand side of \eqref{eq:D1.8}
is the action of $(g,h,p,q)$ on the quadruple
\eqref{eq:D1.5}.)

In a similar way as in the proof of \thmref{thm:T1.3},
we  can show that both functions $g$ and $h$
are  nonnegative a.e.\ in
 their respective domains of definition $\R^d$ and $U^\cm$,
 and that the scalars $p,q$ are nonnegative.
In turn, after setting $u = 0$, $v=0$  and $a=b=0$
 in the inequality \eqref{eq:D1.9} it follows
 (using the fact that $Y$ is a linear space) that
\begin{equation}
\label{eq:D2.1}
\int_{\R^d} \ft{f}  g  - \int_{U^\cm} f h + p   \int_{\R^d} f   - q f(0)  = 0
\end{equation}
for every $f \in Y$. Hence we
 may assume that $c=0$.

By replacing $g(t)$ with $\tfrac1{2} (g(t) + g(-t))$, 
and similarly replacing 
 $h(t)$ with $\tfrac1{2} (h(t) + h(-t))$, we may
assume that both $g$ and $h$
are  nonnegative and even functions
on their respective domains of definition $\R^d$ and $U^\cm$
(we note here that $-U^\cm = U^\cm$).

It will be convenient now to extend $h$
to the whole $\R^d$ by setting $h = 0$ on $U$.
In this case, for every  \emph{Schwartz} function $f \in Y$
we can write \eqref{eq:D2.1} as 
\begin{equation}
\label{eq:D2.2}
 ( \ft{g} -h + p  - q  \del_0) (f) =  0.
\end{equation}
Using the fact that the tempered distribution
$ \ft{g} - h + p  - q \del_0$ is real and even,
the  equality 
\eqref{eq:D2.2} for all Schwartz functions $f \in Y$ 
implies that
$ \ft{g} -h  + p  - q \del_0 = 0$.

Next we show that  the scalars $p,q$ are 
not only nonnegative, but in fact must be strictly
positive. Indeed, from 
 \eqref{eq:D1.8} we obtain that $p> 0$ (since
 $c=0$).   Hence
  $g + p \del_0$ is a nonzero positive measure,
  whose Fourier transform satisfies
$ \ft{g}  + p = q \del_0 + h$, and as a consequence,
$ \ft{g}  + p = q \del_0$ in $U$. But since 
$ \ft{g}  + p$ cannot vanish in any neighborhood
  of the origin, this implies that $q>0$.

Finally, define $\al = q^{-1} ( \ft{g}  + p )$. Then $\al = \del_0 + q^{-1} h$,
the function $q^{-1} h$ vanishes on $U$ and is nonnegative on
$U^\cm$, 
and $\ft{\al}$ is a positive measure. Hence $\al$ is an
admissible tempered distribution for the dual Delsarte problem.
But 
 $\ft{\al} = q^{-1}( p  \del_0 + g)$, so that $\ft{\al}$ has an
 atom at the origin of mass $\ft{\al}(\{0\}) = p/q$. However due to
\eqref{eq:D1.8} and recalling that  $c=0$,
we have $\ft{\al}(\{0\}) > D'(U)$ which gives us the desired
contradiction.
\end{proof}

\subsubsection{Remark}
Again we obtain an interesting fact as a consequence of the proof.
We may consider a smaller class of admissible tempered distributions 
$\al$ by requiring  that 
$\al = \del_0 + h$, where $h \in L^\infty(\R^d)$
is a nonnegative even function which vanishes a.e.\ on $U$, 
and that $\ft{\al}$ is of the form  $\ft{\al} = \ft{\al}(\{0\})  \del_0 + g$, where $g$
is a nonnegative even function in $L^\infty(\R^d)$. Then
the supremum of $\ft{\al}(\{0\}) $ over this smaller class of admissible 
tempered distributions 
$\al$ still gives us the dual Delsarte constant $D'(U)$.

\subsection{Existence of extremizers for the Delsarte problem and its dual}
\label{ss:existence-delsarte}
We say that a Delsarte admissible function $f$
is extremal if it satisfies $\int f = D(U)$.
Similarly, 
a tempered distribution  $\al$ which is 
admissible for the dual Delsarte problem
is called extremal
if $ \ft{\al} (\{0\})  = D'(U)$.

The existence of a Delsarte extremizer 
in the general context of locally compact
abelian groups was proved in \cite{Ram25},
\cite[Corollary 20]{BRR24}.
The existence of an extremizer for the 
dual Delsarte problem is shown in 
\cite[Proposition 3.6]{CLS22}.

\begin{thm}
\label{thm:D8.1}
Let $U \sbt \R^d$ be an open set 
satisfying \eqref{it:dsadmiti}, \eqref{it:dsadmitii}, \eqref{it:dsadmitiii}. Then,
\begin{enumerate-roman}
\item \label{it:dex.1}
The Delsarte problem admits at least
one extremal function $f$;
\item \label{it:dex.2}
The dual Delsarte problem admits at least
one extremal  tempered distribution $\al$.
\end{enumerate-roman}
\end{thm}

\begin{proof}
Again part \ref{it:dex.1} follows 
from \lemref{lem:D1.1} and 
the equality \eqref{eq:D1.3}
of \thmref{thm:D1.3}, which is now proved.
We therefore turn to prove part \ref{it:dex.2}.

Let $\{ \al_j\}$ be a sequence  of
tempered distributions  which are 
admissible for the dual
 Delsarte problem, and such that
  $\ft{\al}_j( \{0 \})  \to D'(U)$.
\lemref{lem:T6.9.1} implies
that $\{\ft{\al}_j\}$ is a 
uniformly translation-bounded
 sequence of positive measures.
 So after passing to a subsequence 
 we may assume that $\ft{\al}_j$ converges
 vaguely  to a certain translation-bounded positive
 measure, which we denote as $\ft{\al}$
 for some
 tempered distribution $\al$.
 The  uniform   translation-boundedness and the
 vague convergence imply that 
 $\ft{\al}_j \to \ft{\al}$ 
in the sense of  tempered distributions.
Hence also $\al_j \to \al$ 
in the sense of  tempered distributions.
Since we have $\al_j = \del_0 +  \be_j$
 where $\be_j$ is a positive measure supported in
the closed set $U^\cm$, it follows that the limiting
tempered distribution $\al$ must also be of the form 
$\al = \del_0 + \be$ for some positive measure $\be$ 
supported in the closed set $U^\cm$. 
Hence  $\al$ is an
admissible tempered distribution
for the dual Delsarte problem.
Moreover,  since the measure $\ft{\al}_j$ 
is  positive and has mass 
$\ft{\al}_j( \{0 \}) $
at the origin, then the vague limit $\ft{\al}$
also has an atom at the origin, of mass at least 
$\lim \ft{\al}_j( \{0 \})  = D'(U)$.
We conclude that 
$\al$ is extremal for the 
dual Delsarte problem.
\end{proof}

 \subsection{Relation between extremizers
 for the Delsarte problem and its dual}\label{ss:properties-delsarte}

Our next goal is to establish   relations between 
extremal functions for the Delsarte problem,
and extremal
tempered distributions for the 
dual Delsarte problem.

 \subsubsection{}
 
We start with a few general observations.
Let $f$ be any function admissible for the Delsarte problem,
and let  $\al$ be any tempered distribution admissible 
for the dual Delsarte  problem (we assume neither
$f$ nor $\al$ to be extremal here).
 
 First,   note that the products $f \cdot \al$ and
$\ft{f} \cdot \ft{\al}$  are both well-defined,
since  $f$ and $\ft{f}$ are
  continuous functions,  while
$\al$ and  $\ft{\al}$ are positive measures.

Next, we recall that the convolution
$\ft{f} \ast \ft{\al}$  is well-defined.
Indeed, $\ft{f}$ is a nonnegative continuous
 function in $L^1(\R^d)$, while 
  $\ft{\al}$ is  a positive translation-bounded
  measure
   (due to \lemref{lem:T6.9.1}),
    hence $\ft{f} \ast \ft{\al}$
is a well-defined, positive translation-bounded measure, 
which is also 
a locally integrable function.

We   claim that also  the convolution $f \ast \al$ 
is well-defined. Indeed, we have:

\begin{lem}
\label{lem:D6.9.1}
Let $\al$ be any tempered distribution  admissible 
for the dual Delsarte problem. Then $\al$ is a
translation-bounded measure.
\end{lem}

\begin{proof}
We choose  and fix a Schwartz function  $\pphi$
satisfying $\ft{\varphi}(t) \geq \1_B(t)$
for all $t \in \R^d$, where
$B$ is the open unit ball. Since $\al$ is a positive
measure, we have
\begin{equation}
\al(B+t)  \leq \int \ft{\pphi}(s-t)  d \al(s) =
  \int e^{2 \pi i \dotprod{t}{x}}  \varphi(x) d\ft{\al}(x),
\end{equation}
and hence
\begin{equation}
\sup_{t \in \R^d} \al(B+t)   \leq 
\int |\varphi(x)|  d\ft{\al}(x) < +\infty,
\end{equation}
where the last integral is finite since 
$\varphi$ has fast
decay, while $\ft{\al}$ is a translation-bounded measure
due to \lemref{lem:T6.9.1}.
 This shows that also $\al$ is translation-bounded. 
\end{proof}

Hence, if $f$ is Delsarte admissible and
$\al$ is admissible for the dual Delsarte  problem,
then $f \in L^1(\R^d)$ while $\al$ 
 is  a positive translation-bounded   measure
 (due to \lemref{lem:D6.9.1}),
 which implies that the convolution $f \ast \al$
 is a well-defined translation-bounded signed measure, 
which is also 
a locally integrable function.

 \subsubsection{}
Now suppose that  $f$ and $\al$ are 
extremizers for the Delsarte problem and 
its dual, respectively. 
This means that 
 $f$ is a Delsarte  admissible function, with $\int f = D(U)$,
while $\al$ is a tempered distribution which is admissible 
 for the dual Delsarte  problem,  
  such that $\ft{\al}(\{0\}) = D'(U)$.

\begin{thm}
\label{thm:D3.9.2}
Let $U \sbt \R^d$ be an open set 
satisfying \eqref{it:dsadmiti}, \eqref{it:dsadmitii}, \eqref{it:dsadmitiii}. 
If $f$ and $\al$ are any two extremals
for the Delsarte problem and
 its dual, respectively, then
\begin{enumerate-roman}
\item
 \label{it:dsrel.1} 
The  
measure $\al$ is supported on the closed set
 $\{x : f(x)=0\} \cup \{0\}$,
 and as  a consequence,  we have
 $f \cdot \al = \del_0$ and $\ft{f} \ast \ft{\al} = 1$ a.e.;
\item
 \label{it:dsrel.2}
The  
measure $\ft{\al}$ is supported on the closed set
 $\{t : \ft{f}(t)=0\} \cup \{0\}$, 
and therefore 
 $\ft{f} \cdot \ft{\al} = \del_0$  and $f \ast \al = 1$ a.e.
\end{enumerate-roman}

\end{thm}

\begin{proof}
If
 $f$ and $\al$ are extremals then we have
$ \ft{\al}(\{0\}) \int f  = D'(U) D(U) = 1$.
Let us recall the proof of \thmref{thm:D1.1},
and examine the circumstances 
 under which the inequality $ \ft{\al}(\{0\}) \int f \le 1$
 becomes an equality.
 
We fix a nonnegative Schwartz function  $\pphi$ with
$\int \pphi = 1$, such that $\ft{\pphi}$ is nonnegative
and has compact support. Let
$\pphi_\eps(t) = \eps^{-d} \pphi(t/\eps)$, then 
$\ft{\pphi}_\eps(x)  = \ft{\pphi}(\eps x)$.
We can write 
$\ft{\al} =  D'(U)  \del_0 + \mu$,
where $\mu$ is a positive measure. Hence
\begin{equation}
\label{eq:D7.2}
(\ft{f} \ast \ft{\al} )(\pphi_\eps ) 
 = D'(U) \int  \ft{f} \cdot \pphi_\eps    + 
 \int 
(\ft{f} \ast \mu ) \cdot \pphi_\eps.
\end{equation}
On the other hand,
$\al = \del_0 + \be$  where
$\be$ is a positive measure supported on $U^\cm$. 
Recalling that the Fourier transform  of 
$f \cdot \al$  is $\ft{f} \ast \ft{\al}$,
due to \lemref{lemB1},
and since we have
$f(0) \ft{\pphi}_\eps (0) = 1$, this implies that
\begin{equation}
\label{eq:D7.3}
(\ft{f} \ast \ft{\al} )(\pphi_\eps ) =
(f \cdot \al)(\ft{\pphi}_\eps ) =
1 + \int_{U^\cm} \ft{\pphi}_\eps(x) f(x) d \beta(x)
\end{equation}
(we also recall that the  integral 
on the right hand side of \eqref{eq:D7.3}
is well-defined, since 
$\ft{\pphi}_\eps$ has compact support).

We now claim that the measure $\beta$ is supported
on the closed set $\{x : f(x) = 0\}$, 
and that the measure $\mu$ is supported
on the closed set $\{t : \ft{f}(t) = 0\}$.

To  prove these two claims, we first note that 
$\be$ is a positive measure supported on
$U^\cm$, while $f$ is a continuous function which
is nonpositive on $U^\cm$. Hence
$f \cdot \beta$ is a nonpositive measure.
Since $\ft{\pphi}_\eps$ is a nonnegative function,
this implies that 
the right hand side of 
\eqref{eq:D7.3} cannot exceed $1$. Moreover,
since $\ft{\pphi}_\eps \to 1$ locally uniformly as $\eps \to 0$, 
it follows that 
if $\beta$ is not supported
on the closed set $\{x : f(x) = 0\}$, 
then the $\limsup$ as $\eps \to 0$  of the right hand side of 
\eqref{eq:D7.3}
must be strictly smaller than $1$.

Next, recall that
 $\ft{f} \ast \mu$ is a positive
translation-bounded measure,
 which is also a locally integrable function.
Since $\pphi_\eps$ is  nonnegative, then
 $\int (\ft{f} \ast \mu) \cdot \pphi_\eps$  is 
nonnegative; and since  $\int \ft{f} \cdot \pphi_\eps \to \ft{f}(0) = D(U)$
 as $\eps \to 0$, 
the $\liminf$ as $\eps \to 0$   of the right hand side of 
\eqref{eq:D7.2} cannot
be  less than $D'(U)D(U) = 1$.
Moreover, if $\mu$ is not supported
on the closed set $\{t : \ft{f}(t) = 0\}$ then,
in a similar way to the proof of \thmref{thm:T3.7},
it can be shown that the function
$\ft{f} \ast \mu$ is a.e.\ not less than
some nonnegative continuous 
function whose value at the origin is 
 strictly  positive.   This implies that 
the $\liminf$  as $\eps \to 0$ of the right hand side of 
\eqref{eq:D7.2}
must be strictly larger than $1$.
 
 However, since \eqref{eq:D7.2} and \eqref{eq:D7.3} are equal,
  this implies that indeed $\beta$  is supported
on the closed set $\{x : f(x) = 0\}$, 
and $\mu$ is supported
on the closed set $\{t : \ft{f}(t) = 0\}$,
for otherwise this would lead to a contradiction.

We conclude that 
$f \cdot \al = \del_0$ and $\ft{f} \cdot \ft{\al} = \del_0$.
Finally, since both $f$ and $\ft{f}$ are in $L^1(\R^d)$, 
and both $\al$ and $\ft{\al}$ are translation-bounded measures,
this implies using \lemref{lemB1} 
that $\ft{f} \ast \ft{\al} = 1$ a.e.,
 and $f \ast \al = 1$ a.e.
Thus both  assertions 
\ref{it:dsrel.1} and \ref{it:dsrel.2} 
are proved.
 \end{proof}

% =========================================================

\section{The Delsarte packing bound, tiling and spectrality}
\label{s:packing-tiling}

\subsection{The essential difference set}
Let $A \subset \R^d$ be a bounded measurable set of positive measure. The set
\begin{equation}
	\label{eqC2.1}
	\Delta(A) := \{t \in \R^d: m(A \cap (A + t)) > 0\}
\end{equation}
is called the \emph{essential difference set} of $A$.
It is a bounded origin-symmetric open set, that
serves as the measure-theoretic analog of the algebraic
difference set $A-A$.  In particular, if $A$ is an 
open set, then $\Delta(A) = A-A$.

In this section we connect packing, tiling and spectrality
properties of a  bounded  measurable  set $A \subset \R^d$
of positive measure,
to the  Delsarte problem for the 
essential difference set  $U = \Del(A)$.
We observe that this set $U$ 
satisfies \eqref{it:dsadmiti} and \eqref{it:dsadmitii}, 
but not necessarily \eqref{it:dsadmitiii},
as the example $A = (0,1) \cup (2,3) \sbt \R$ shows.

The Delsarte constant $D(\Del(A)) $ 
 of the  set $U = \Delta(A)$ satisfies
 \begin{equation}
\label{eqC2.8}
D(\Del(A))  \ge m(A),
\end{equation}
since the function
 $f = m(A)^{-1} \1_A \ast \1_{-A}$ is 
 Delsarte admissible, and $\int f = m(A)$.

\subsection{Packing}
\label{ss:packing}

If $\Lam \sbt \R^d$ is a countable set, then we say
that $A + \Lam$ is a \emph{packing} if the translated copies
$A + \lam$, $\lam \in \Lam$, are pairwise disjoint
up to measure zero.

Cohn and Elkies \cite{CE03} used the Delsarte
problem (not using this name though) as a method 
for obtaining
upper bounds for the density of sphere packings, or more generally,
packings by translates of a convex, centrally symmetric body
$A \sbt \R^d$, see \cite[Theorem B.1]{CE03}.

The following theorem extends the result to the
case where $A \sbt \R^d$ is a 
general bounded measurable set
(see also \cite[Theorem 3]{KR06} and \cite[Theorem 1.1]{BR23}).

\begin{thm}
\label{thm:ce03}
Let $A \sbt \R^d$ be a bounded measurable set of positive
measure. Then any packing by translates of $A$ has
density not exceeding $D(\Del(A))^{-1}$, that is,
the reciprocal of the Delsarte constant of the set $U = \Del(A)$.
\end{thm}

Hence, any Delsarte admissible
(with respect to the set $U = \Del(A)$) function
$f$, yields an upper bound $(\int f)^{-1}$
for the density of any packing by translates of $A$.
As an example, the function
$f = m(A)^{-1} \1_A \ast \1_{-A}$ 
yields the trivial volume bound $m(A)^{-1}$.

Note that the proof of \thmref{thm:ce03},
as well as Theorems \ref{thm:eutile} and \ref{thm:euspectral} 
below, only relies  on \thmref{thm:D1.1} which
does not require the condition  \eqref{it:dsadmitiii}.

\begin{proof}[Proof of \thmref{thm:ce03}]
It is well known and not hard to show that periodic packings
 come arbitrarily close to
the greatest packing density, see 
\cite[Appendix A]{CE03}. Hence it suffices to prove
that if $\Lam$ is periodic and $A + \Lam$ is a packing,
then the density of $\Lam$ cannot exceed $D(\Del(A))^{-1}$.

Indeed, if $\Lam$ is periodic then 
we may write $\Lam = L + F$ where
$L$ is a lattice and $F$ is a finite set such that 
$(F-F) \cap L = \{0\}$. Then the measure 
$\gam = |F|^{-1} \del_F \ast \del_{-F} \ast \del_L$ 
is  translation-bounded, positive and
supported on $F-F + L  = \Lam - \Lam$,
which is a subset of $\{0\} \cup \Del(A)^\cm$ by
the assumption that  $A + \Lam$ is a packing. Moreover, $\gam$ has a
unit mass at the origin and 
$\ft{\gam} = |F|^{-1} |\ft{\del}_F|^2 \cdot \ft{\del}_L$
is a positive measure, so $\gam$ is admissible
for the dual Delsarte problem. We now observe that
$\ft{\gam}$ has an atom at the origin
of mass $|F| \cdot (\det L)^{-1}$ which is exactly the
density of $\Lam$. This shows that the dual Delsarte constant
$D'(\Del(A))$ must be  at least as large as the
density of $\Lam$. By the weak linear duality  
inequality \eqref{eq:D1.1} this implies that
the density  of $\Lam$ cannot exceed $D(\Del(A))^{-1}$.
\end{proof}

Viazovska \cite{Via17} proved that if $A$ is the open
unit ball in $\R^8$, then the Delsarte bound $D(\Del(A))^{-1}$ 
coincides with the density of the $E_8$-lattice packing,
 showing that this is the
densest sphere packing in dimension $8$.
A similar result was subsequently proved also in dimension $24$,
see \cite{CKMRV17}. However, note that
the Delsarte bound $D(\Del(A))^{-1}$ is not expected
to yield the sharp estimate for sphere packing density in every
dimension, and moreover, for certain dimensions 
the Delsarte bound is known to be not sharp
(see \cite{CDV24} and the references therein).

\subsection{Tiling}
\label{ss:tiling}

We say that a  bounded measurable set 
$A \sbt \R^d$ \emph{tiles
by translations}, if there is a
countable set $\Lam \sbt \R^d$ such that
the translated copies
$A + \lam$, $\lam \in \Lam$, cover the whole
space without overlaps up to measure zero. 

The following result is a direct consequence of \thmref{thm:ce03} above
(see also \cite[Section 3.4]{KR06} and \cite[Proposition 5.7]{BR23}).

\begin{thm}
\label{thm:eutile}
If a bounded measurable set $A \sbt \R^d$
 tiles the space by translations, then we have $D(\Del(A)) = m(A)$.
\end{thm}

\begin{proof}
Indeed, a tiling by translates of $A$ 
is a packing of density $m(A)^{-1}$. 
Hence, using \thmref{thm:ce03}
we obtain that
$m(A)^{-1}$ does not exceed $D(\Del(A))^{-1}$.
But we also have the converse inequality
\eqref{eqC2.8}, so we conclude that the equality
 $D(\Del(A)) = m(A)$ holds.
 \end{proof}

\subsection{Spectrality}
\label{ss:spectrality}

A bounded, measurable set $A \sbt \R^d$ is called 
\emph{spectral} if the space $L^2(A)$ has an orthogonal basis 
consisting of  exponential functions.
Fuglede  \cite{Fug74} famously conjectured  that $A$ is a spectral set
if and only if it can tile the space by translations.
This conjecture inspired extensive research over the years,
see \cite{Kol24} for the history of the
problem and an overview of the known
related results.

The following result is analogous to 
\thmref{thm:eutile}, but with the tiling assumption
being replaced by spectrality
(see also \cite[Theorem 5]{KR06}).

\begin{thm}
\label{thm:euspectral}
If a bounded measurable set $A \sbt \R^d$
is spectral, then $D(\Del(A)) = m(A)$.
\end{thm}

\begin{proof}
This is a consequence of a result proved in
\cite[Theorem 3.1]{LM22}. Stated
using the terminology of the present paper,
the result asserts that if $A$ is spectral,
then there exists 
a tempered distribution  $\al$ on $\R^d$,
which is admissible for the dual Delsarte problem
with respect to  the set $U = \Del(A)$, and such that
$\ft{\al}(\{0\}) = m(A)^{-1}$. 

This result thus implies that $D'(\Del(A)) \ge m(A)^{-1}$.
In turn, as a consequence of the weak linear duality  
inequality \eqref{eq:D1.1}, it follows that 
$D(\Del(A)) \le m(A)$. As before, this
 suffices to conclude the proof, since the
 converse  inequality \eqref{eqC2.8} also holds.
\end{proof}

\subsection{Convex bodies}
\label{ss:convex}

By a \emph{convex body} $A \sbt \R^d$ we mean 
 a compact convex set with nonempty interior.
 A major recent result proved in \cite{LM22}
  states that the Fuglede conjecture 
holds for convex bodies in all dimensions.
That is, a convex body $A \sbt \R^d$
 is  a spectral set if and only if it can tile
the space by translations.

If $A \sbt \R^d$ is a convex body
then the set $ \Del(A)$ is convex.
Moreover, if the convex body $A$ is origin-symmetric,
then $\Del(A)$ is equal to the interior of the set $2 A$.

The next result shows that for a convex body, the converse to
 Theorems \ref{thm:eutile} and \ref{thm:euspectral}
 is also true: the condition $D(\Del(A)) = m(A)$
in fact characterizes the
convex bodies $A \sbt \R^d$ which
tile the space by translations (or equivalently, which are spectral).

\begin{thm}
\label{thm:C1.2}
Let $A \sbt \R^d$  be a convex body.
The equality $D(\Del(A)) = m(A)$ holds if and only if $A$ tiles
 by translations.
\end{thm}

This result has an interesting consequence for packing
density estimates:

\begin{cor}
If a convex body $A \sbt \R^d$ does not tile the space,
then the Delsarte bound 
$D(\Delta(A))^{-1}$
for the greatest packing density  
provides a strictly better estimate 
than the trivial volume packing bound $m(A)^{-1}$.
\end{cor}

It follows that if $A$ does not tile, then there is always
a nontrivial packing density estimate 
(i.e.\ better than the volume packing bound $m(A)^{-1}$) of the form  $(\int f)^{-1}$ 
for some appropriately chosen Delsarte admissible function $f$.

The proof of \thmref{thm:C1.2} is based on the connection of the problem
to the concept of weak tiling, 
 introduced in \cite{LM22}
 as a relaxation of proper tiling.
We say that a bounded, measurable set $A \sbt \R^d$
\emph{weakly tiles its complement} 
if there exists a positive, locally finite measure $\nu$
such that $\1_{A} \ast \nu = \1_{A^\cm}$ a.e.
This notion generalizes proper tilings
which correspond to the case 
where the measure $\nu$ is a sum of unit masses.

The notion of weak tiling played a key role in the proof
of Fuglede's conjecture  for convex domains, due to the 
 fact that every spectral
set must weakly tile its complement, see \cite[Theorem 1.5]{LM22}.
This result can be viewed as a weak form of the
``spectral implies tiling'' part of  Fuglede's conjecture.
Note that generally, a set that weakly tiles
its complement need not tile properly
(as an example, take any spectral set which does not tile).

However, it was proved in \cite[Theorem 1.4]{KLM23} that for 
a convex body, weak tiling implies tiling.
More precisely,   if a convex body
$A \sbt \R^d$ weakly tiles its complement, then $A$ 
must be a convex polytope which can also tile  properly by translations.
The proof is composed of several results obtained in different
papers, see 
\cite[Section 2]{KLM25} for an overview of the
 ingredients of the proof, the relevant references, and
 a simplification of part of the proof.

We  now turn to show that the condition $D(\Del(A)) = m(A)$
indeed characterizes the convex bodies $A \sbt \R^d$ which
tile  by translations.
Note that if $A$ is a convex body,
then the set $U  = \Del(A)$ 
is a bounded origin-symmetric open convex set, 
and hence satisfies all the three conditions
\eqref{it:dsadmiti}, \eqref{it:dsadmitii}, \eqref{it:dsadmitiii}.

\begin{proof}[Proof of \thmref{thm:C1.2}]
We  already know that
if  $A$ tiles by translations, then the equality
 $D(\Del(A)) = m(A)$ holds (\thmref{thm:eutile}).
We need to prove the converse direction.

Assume that  $D(\Del(A)) = m(A)$. Then  the
function  $f = m(A)^{-1} \1_A \ast \1_{-A}$ is extremal
for the Delsarte problem with $U = \Del(A)$. 
By \thmref{thm:D8.1}\ref{it:dex.2},
the dual Delsarte problem also admits
at least one extremal $\al$.
The strong linear duality equality
\eqref{eq:D1.3} gives us that
$\ft{\al}(\{0\}) = D'(\Del(A)) = m(A)^{-1}$.
Moreover, by  \thmref{thm:D3.9.2}\ref{it:dsrel.2},
 the measure $\ft{\al}$ must be supported on the set
 $\{t : \ft{f}(t)=0\} \cup \{0\}$.
 However, since  $\ft{f} = m(A)^{-1} |\ft{\1}_A |^2$, 
 the two functions $\ft{f}$ and $\ft{\1}_A$ have
 the same set of zeros.  It thus follows that
$\ft{\1}_A \cdot \ft{\al} = \del_0$.
By \lemref{lem:D6.9.1}, the measure
$\al$ is translation-bounded, so
we may use \lemref{lemB1} to conclude 
that $\1_A \ast \al = 1$ a.e. 
In turn, since we have 
$\al = \del_0 + \be$  for some
positive measure $\be$, this implies that 
$\1_A \ast \be = \1_{A^\cm}$
a.e., that is, $A$ weakly tiles
its complement.
Finally,  \cite[Theorem 1.4]{KLM23} yields that $A$ can also tile properly
by translations.
\end{proof}

\subsection{Tur\'{a}n domains}
Lastly, we mention another interesting consequence of Theorem \ref{thm:C1.2} in  relation to the possible existence of non-Tur\'an domains. 

First we recall that for a convex bounded 
origin-symmetric open set $U \sbt \R^d$ we have

\begin{equation}
D(U) \ge T(U) \ge 2^{-d} m(U).
\end{equation}
The last inequality is true since the closure of
the set $\frac1{2}U$ is a convex body $A$
satisfying $\Del(A) = U$ and $m(A) = 2^{-d} m(U)$,
and so the function
 $f = m(A)^{-1} \1_A \ast \1_{-A}$ is 
 Tur\'{a}n admissible and satisfies
$\int f = 2^{-d} m(U)$.

As an example, let $U \sbt \R^d$ be an open ball 
centered at the origin. In this case, it
 is known that  $T(U) = 2^{-d} m(U)$, see \cite{Gor01}, \cite{KR03}, \cite{Gab24}.
On the other hand, in dimensions greater than one,
a ball cannot tile by translations. Hence \thmref{thm:C1.2} yields that
$D(U) > T(U) = 2^{-d} m(U)$, that is, the Delsarte constant of a ball
is strictly greater than the corresponding Tur\'{a}n constant.

Recall that it
is not known whether there exists a convex
bounded origin-symmetric open set $U \sbt \R^d$
which is not a Tur\'an domain, i.e.\ such that 
$T(U) > 2^{-d} m(U)$.
The following result gives a possible line
of attack for constructing a non-Tur\'an domain.

\begin{cor}
Assume that there exists a convex
bounded origin-symmetric open set $U \sbt \R^d$,
which does not tile, but $T(U)=D(U)$.
Then $U$ is a non-Tur\'an domain.
\end{cor}

\begin{proof}
Indeed, if $U$ does not tile then by
\thmref{thm:C1.2} we have
 $D(U) > 2^{-d} m(U)$, so using
 the assumption that $T(U) = D(U)$
this implies that $U$ is a non-Tur\'an domain.
\end{proof}

In conclusion, for a convex body $A \sbt \R^d$, 
it is instructive to inspect the chain of inequalities 
\begin{equation}
m(A)\le 
2^{-d} m(\Delta(A))\le T(\Delta(A))
\le D(\Delta(A)).
\end{equation}
 If $A$ is non-symmetric then the first inequality is strict by the Brunn-Minkowski inequality. If $A$ is symmetric but does not tile the space, then the first inequality
 becomes  an equality, but at least one of the other two inequlities must be strict by Theorem \ref{thm:C1.2}. 
 Finally, if $A$ tiles the space, then all inequalities become equalities.

% =========================================================

\section{Remarks}\label{s:remarks}

\subsection{}
Some authors define the Tur\'{a}n constant 
alternatively as the supremum
of $\int f$ over all  the continuous
real-valued functions $f$, 
\emph{with compact support contained in the open set $U$},
such that
$f(0) = 1$ and $\ft{f}$ is nonnegative.
We denote this supremum by $T_0(U)$.

It is obvious that $T_0(U) \le T(U)$. The question
whether the equality 
$T_0(U) = T(U)$ holds or not, seems to be 
largely open. It is easy to show that the equality
holds  if $U \sbt \R^d$ is a convex bounded 
origin-symmetric open set, or more generally, 
if  $U \sbt \R^d$ is a bounded  origin-symmetric  open
set  assumed to be
\emph{strictly star-shaped} (with respect to the origin),
which by definition means that the closure of
$\lam U$ is contained in $U$ for every $0 \le \lam < 1$,
see e.g.\ \cite[Theorem 1]{Mav13}.

Another case where the equality
is known to hold is when $U \sbt \R$ is a 
bounded open set, $0 \in U = - U$,
composed of finitely
many intervals, see \cite[Theorem 2]{Mav13}.

No example seems to be known
of an open set $U \sbt \R^d$ of finite measure,
$0 \in U = -U$, such that $T_0(U) < T(U)$.

See also \cite[Section 5]{BRR24} where the
question is discussed 
in the general context of locally compact
abelian groups.

\subsection{}
The Tur\'{a}n problem and its dual were also considered by
Gabardo in \cite{Gab24}.  The main result 
\cite[Theorem 4]{Gab24} concerns the case
where $U$ is an open ball
(centered at the origin)  in $\R^d$.
In this case, Gabardo constructed a
tempered distribution $\al$ which is admissible 
for the dual Tur\'{a}n problem, such that
$\ft{\al}(\{0\}) = 2^{d} m( U )^{-1}$.
This allowed him to conclude that 
$\al$ is extremal for the dual Tur\'{a}n problem,
and to obtain a new proof of the fact that
 if $U$ is an open ball in $\R^d$ then
 $T(U) = 2^{-d}  m( U)$. It is also noted 
\cite[Proposition 18]{Gab24} that
this extremal tempered distribution $\al$ 
is \emph{not a measure}.

In \cite{Gab19}, \cite{Gab20}, Gabardo announced
a result concerning the factorization of positive definite functions through
convolutions in locally compact abelian groups, 
which implies that the strong linear duality equality
 $T(U)T'(U)=1$ holds for any open set 
$U \sbt \R^d$ of finite measure, $0 \in U = -U$.
 To our knowledge, the proof
remains unpublished.

% =========================================================

\end{document}